\documentclass[11pt,a4paper]{amsart}
\usepackage[centering, margin=1in]{geometry}
\usepackage[latin1]{inputenc}
\usepackage[english]{babel}
\usepackage[T1]{fontenc}
\usepackage{amsmath}
\usepackage{amssymb}
\usepackage{amsthm}
\usepackage{bm}
\usepackage{array,booktabs}
\usepackage{graphicx}
\usepackage{pagecolor}
\usepackage{lmodern}
\usepackage{tcolorbox}
\usepackage{quoting}
\usepackage{pgfplots}
\usepgfplotslibrary{polar}
\pgfplotsset{compat=newest}
\usepackage{enumerate}
\usepackage{mathtools}
\usepackage{csquotes}
\usepackage{cite}
\usepackage{tikz}
\usepackage{tikz-cd}
\usetikzlibrary{matrix,arrows,decorations.pathmorphing}
\usepackage{wrapfig}
\renewcommand{\Re}{\operatorname{Re}}
\renewcommand{\Im}{\operatorname{Im}}
\def\R{\ensuremath\mathbb{R}}
\def\C{\ensuremath\mathbb{C}}
\def\A{\ensuremath\mathbb{A}}
\def\Z{\ensuremath\mathbb{Z}}
\def\Q{\ensuremath\mathbb{Q}}

\def\Hb{\ensuremath\mathbb{H}}

\usepackage{hyperref}
\usetikzlibrary{arrows}
\usepackage{amsmath}
\usepackage{mathtools}
\usepackage{adjustbox}
\usepackage[final]{pdfpages}
\usepackage{verbatim}
\newtheorem{thm}{Theorem}[section]
\newtheorem{defi}[thm]{Definition}
\newtheorem{cor}[thm]{Corollary}
\newtheorem{lemma}[thm]{Lemma}
\newtheorem{prop}[thm]{Proposition}
\theoremstyle{remark}
\newtheorem{remark}[thm]{Remark}

\def\eps{\ensuremath\varepsilon}
\def\0{\emptyset}

\def\Cl{\text{\rm Cl}}

\def\sym{\mathrm{sym}}

\def\SL{\hbox{\rm SL}}
\def\PSL{\mathrm{PSL}}
\def\PGL{\mathrm{PGL}}
\def\GL{\mathrm{GL}}
\def\modulo{\text{ \rm mod }}

\numberwithin{equation}{section}

\numberwithin{equation}{section}
\begin{document}
\title[Closed geodesics in homology]{Concentration of closed geodesics in the homology of modular curves}
\author{Asbj\o rn Christian Nordentoft}

\address{LAGA, Institut Galil\'{e}e, 99 avenue Jean Baptiste Cl\'{e}ment, 93430 Villetaneuse, France}

\email{\href{mailto:acnordentoft@outlook.com}{acnordentoft@outlook.com}}

\date{\today}
\begin{abstract}
We prove that the homology classes of closed geodesics associated to subgroups of narrow class groups of real quadratic fields concentrate around the Eisenstein line. This fits into the framework of Duke's Theorem and can be seen as a real quadratic analogue of results of Michel and Liu--Masri--Young on supersingular reduction of CM-elliptic curves.  We also study the level aspect, as well as a homological version of the sup norm problem. Finally we present applications to group theory and modular forms.\end{abstract}
\maketitle
\section{Introduction}
Let $p$ be a prime and consider the modular curve $Y_0(p)=\Gamma_0(p)\backslash \Hb$ of level $p$ equipped with the hyperbolic line element $|dz|/y$ and volume element $dxdy/y^2$, where $\Hb=\{z=x+iy\in \C: y>0\}$ is the upper half-plane and $\Gamma_0(p)\leq \PSL_2(\Z)$ is the level $p$ Hecke congruence group\footnote{Here and throughout we will (contrary to standard conventions) consider the Hecke congruence groups $\Gamma_0(p)$ as subgroups of $\PSL_2(\R)$.} acting on $\Hb$ by linear fractional transformations. Given a real quadratic field $K$ of discriminant $d_K >0$ such that $p$ splits in $K$, one can associate to an element $A\in \Cl_K^+$ of the narrow class group of $K$ an oriented closed geodesic $\mathcal{C}_{A}(p)$ on $Y_0(p)$ (see Section \ref{sec:geodesic} for details). 
In a celebrated paper \cite{Du88} Duke proved that the geodesics $\{\mathcal{C}_{A}(p)\subset Y_0(p): A\in \Cl_K^+\}$ equidistribute with respect to hyperbolic measure as $d_K$ tends to infinity, meaning that
\begin{equation}\label{eq:DukeTHM} \frac{\sum_{A\in \Cl_K^+} \int_{\mathcal{C}_{A}(p)} f(z)\tfrac{|dz|}{y}}{\sum_{A\in \Cl_K^+} \int_{\mathcal{C}_{A}(p)} 1 \tfrac{|dz|}{y}}\rightarrow \frac{1}{\mathrm{vol}(Y_0(p))}\int_{Y_0(p)} f(z)\tfrac{dxdy}{y^2}, \quad d_K\rightarrow \infty, \end{equation}
for $f:Y_0(p)\rightarrow \C$ smooth of compact support (for further results see \cite{EinLindMichVenk12}, \cite{MichelVenk06} and the reference therein). In this paper we study the \emph{homological behavior} of the oriented closed geodesics, i.e. the map
\begin{equation}\label{eq:maphom} \Cl_K^+\rightarrow H_1(Y_0(p),\Z), \quad A\mapsto [\mathcal{C}_{A}(p)]:=\text{class of }\mathcal{C}_{A}(p),\end{equation}
from the narrow class group to the integral homology of $Y_0(p)$ (which one can identity with the abelinization of $\Gamma_0(p)$ modulo torsion) as $d_K \rightarrow \infty$. Our result can be stated as saying that the classes concentrate around the \emph{Eisenstein line}. This is a real quadratic analogue of the equidistribution of supersingular reduction of CM elliptic curves as in \cite{Michel04} (see Section \ref{sec:CM}). We also study the level aspect (analogue of \cite{LMY15}) leading to a homological version of the \emph{sup norm problem} (see Section \ref{sec:amplpretrace}) which might be of independent interest (for another application consult \cite{HumphriesNordentoft22}). We also present applications of our distribution results; one group theoretic and another concerning non-vanishing of cycle integrals of modular forms (see Section \ref{sec:appl}). Finally we refer to Section \ref{sec:geoint} and Remark \ref{rem:risager} below for geometric interpretations of our results.
\subsection{Statement of results}\label{sec:state}
Let $p$ be prime and let $g$ denote the genus of the (non-compact) Riemann surface $Y_0(p)$ satisfying $g=\tfrac{p}{12}+O(1)$. The integral homology 
$$H_1(Y_0(p),\Z)\cong \Z^{2g+1},$$ 
sits as a lattice inside the real homology 
\begin{equation}\label{eq:VpGab}V_p:=H_1(Y_0(p),\R)\cong \Gamma_0(p)^\mathrm{ab}\otimes \R \cong \R^{2g+1}.\end{equation}
We will be interested in how elements of $V_p$ distribute when projected to the $2g$-sphere by which we mean the map 
\begin{equation}V_p-\{0\}\twoheadrightarrow\mathbf{S}(V_p):=(V_p-\{0\})/\R_{>0} , \quad v\mapsto \overline{v}.\end{equation}
where we endow $\mathbf{S}(V_p)$ with the quotient topology of the Euclidean topology of $V_p$. 
We define the \emph{Eisenstein class}; 
\begin{equation}v_{E}(p)\in H_1(Y_0(p),\Z),\end{equation} 
as the homology class of a simple loop going once around the cusp at $\infty$ with positive orientation (which corresponds to the image of the matrix $T=\begin{psmallmatrix}1& 1\\ 0 & 1 \end{psmallmatrix}$ in $\Gamma_0(p)^\mathrm{ab}$ using the identification (\ref{eq:VpGab})). This is indeed a Hecke eigenclass with the same eigenvalues as the weight $2$ Eisenstein series as we will see in Section \ref{sec:cohomH}. Our first main result is the following. 
\begin{thm}\label{thm:std3}
Let $p$ be a prime and consider a real quadratic field $K$ of discriminant $d_K $ such that $p$ splits in $K$ with $p\mathcal{O}_K=\mathfrak{p}_1\mathfrak{p}_2$. Consider a subgroup $H\leq \Cl_K^+$ such that $\mathfrak{p}_1 \notin H$ and $(\sqrt{d_K})\notin H$. 

Then as $d_K \rightarrow \infty$, the classes of the closed geodesics associated to $H$ concentrate around the line generated by the Eisenstein element, meaning that 
\begin{equation}\label{eq:main:std3}  \overline{\sum_{A\in H} [C_A(p)]} \longrightarrow \overline{- v_{E}(p)} , \quad \text{as }d_K \rightarrow \infty, \end{equation}
in the quotient topology of $\mathbf{S}(V_p)$. 
\end{thm}
This is the real quadratic analogue of a distribution result due to Michel \cite{Michel04} (see also \cite{ElkiesOnoYang05}, \cite{Yang08}, \cite{Kane09}, \cite{LMY15}, \cite{AkaLuethiMichelWieser22}) concerning the map 
\begin{equation}\label{eq:michelCM} \Cl_K\rightarrow \mathcal{E\ell\ell}^{ss}(\mathbb{F}_{p^2}),\end{equation} 
from the class group of an imaginary quadratic field $K$ (in which $p$ is inert) to the isomorphism classes of supersingular elliptic curves defined over $\mathbb{F}_{p^2}$ (or equivalently over $\overline{\mathbb{F}_p}$). We refer to Section \ref{sec:CM} for a more elaborate explanation of the analogy between the two cases. We note also that our results can be phrased as a weak convergence statement (see Theorem \ref{thm:std}) which resembles (\ref{eq:DukeTHM}). Another useful analogue to have in mind is the distribution of lattice points on the unit sphere $S^2\subset \R^3$ \cite{Du88}:
$$\frac{1}{\sqrt{d}}\{(a,b,c)\in \Z^3: a^2+b^2+c^2=d\}\subset S^2,$$
as $d\rightarrow \infty$ (which is the basic case of Linnik's Problem, see \cite{MichelVenk06}); in both cases one rescales the points to get a convergences of measures to respectively, the Haar measure and the point measure at (minus) the Eisenstein element.

We have the following useful corollary. In Section \ref{sec:geoint} below we will use this to explain the geometric content of our result. 
\begin{cor}
Let $K$ and $H\leq \Cl_K^+$ be as in Theorem \ref{thm:std3} and consider a basis $B$  of $ V_p$ containing $v_{E}(p)$. Then for $d_K$ sufficiently large, the $v_{E}(p)$-coordinate of the vector
$$\sum_{A\in H} [\mathcal{C}_{A}(p)]\in V_p,$$ 
in the basis $B$ has strictly maximal absolute value among all coordinates (and in particular is non-zero).    
\end{cor}

\begin{remark}
The conditions on the level and discriminant in the theorems above are necessary. If either $p\mathcal{O}_K=\mathfrak{p}_1\mathfrak{p}_2$ with $\mathfrak{p}_1 \in (\Cl_K^+)^2$ or $(\sqrt{d_K})\in H$ then there is a basis for $V_p$ (the Hecke basis) such that the $v_{E}(p)$-coordinate of 
$$  \sum_{A\in (\Cl_K^+)^2} [\mathcal{C}_{A}(p)], $$
is zero. In Section \ref{sec:geodesic} we construct for each $p$ an infinite family of real quadratic fields $K$ that satisfy the conditions in Theorem \ref{thm:std3}. It is unclear whether the statement should be true for any genus, say.

\end{remark}
\begin{remark}
The restriction to prime level ensures that there are no old forms and also that there is a unique Eisenstein class in the (co)homology. The main steps in the proofs should however work for general (square-free) level, but the statements would have to be modified accordingly. 
\end{remark}


\subsubsection{Varying the level}
Our second result is concerned with the level aspect in the sense that we will obtain a distribution statement uniform in $p$. Notice first of all that for $v_0,v_1, \ldots\in  V_p-\{0\} $, the convergence 
$$\overline{v_n}\longrightarrow \overline{v_0}\in \mathbf{S}(V_p)\text{ as } n\rightarrow \infty,$$
is equivalent to 
\begin{equation}\label{eq:convergenceformulation} \left|\! \left|\frac{v_n}{|\!|v_n|\!|}-\frac{v_0}{|\!|v_0|\!|}\right|\! \right| \rightarrow 0\text{ as } n\rightarrow \infty,\end{equation}
for any (fixed) norm $|\!|\cdot|\!|$ of $V_p$. A natural question is to ask for bounds for the left hand side of (\ref{eq:convergenceformulation}) \emph{uniform} in $d_K$ and $p$ for certain specific norms of $V_p$. A basis $B$ for $H_1(Y_0(p),\R)$ defines an isomorphism of vectors spaces $V_p\cong \R^{2g+1}$ (by mapping $B$ to the standard basis of $\R^{2g+1}$) and by pulling back the $L^r$-norm for $1\leq r \leq \infty$ we obtain a norm on $V_p$ which we denote by $|\!|\cdot|\!|_{B,r}$ (see (\ref{eq:normsB}) for details). We notice that one can choose a sequence of bases for each $p$ such that the convergence in (\ref{eq:convergenceformulation}) (with say $r=1$) is arbitrarily slow in $p$. This is parallel to the case of the distribution of CM points on modular curves in the level aspect as considered in \cite{LiuMasriYoung13}; here one has to consider {\lq\lq}compatible{\rq\rq} test functions as the level varies. Similarly, we will consider certain {\lq\lq}compatible{\rq\rq} bases of $H_1(Y_0(p),\R)$.  To define these we recall (see Section \ref{sec:zagier}) that the following matrices generate $\Gamma_0(p)$; 
\begin{equation}\label{eq:S}\mathcal{S}(p):=\left\{\begin{psmallmatrix} 1 & 1 \\0 & 1 \end{psmallmatrix}\right\}\cup \left\{\begin{psmallmatrix} a & -(aa^{\ast} +1)/p \\p & -a^\ast \end{psmallmatrix}: 0< a<p, \right\},\end{equation} 
where $0< a^\ast <p$ is such that $aa^{\ast} \equiv -1\modulo p$. We say that a basis $B$ of $H_1(Y_0(p), \R)$ is a \emph{basic basis of level $p$} if it consists of homology classes containing the oriented geodesic connecting $i\in \Hb$ and $\sigma i\in \Hb$ for some $\sigma\in \mathcal{S}(p)$. We think of these bases as analogues of the sets $\mathcal{E\ell\ell}^{ss}(\mathbb{F}_{p^2})$ considered in the imaginary quadratic case (\ref{eq:michelCM}) (see Section \ref{sec:CM}). Our second main result is the following.
\begin{thm}\label{thm:main} Let $p$ be prime and ${B}$ a basic basis of level $p$ with associated norm $|\!|\cdot|\!|_{{B},\infty}$ (note that $|\!|v_{E}(p)|\!|_{B,\infty}=1$). Let $K$ be a real quadratic field of discriminant $d_K $ with no unit of norm $-1$ such that $p$ splits in $K$ with $p\mathcal{O}_K=\mathfrak{p}_1\mathfrak{p}_2$ and $\mathfrak{p}_1\notin (\Cl^+_K)^2$.  Then we have 
\begin{equation}\label{eq:thm1}  \left|\! \left|\frac{\sum_{A\in (\Cl^+_K)^2} [\mathcal{C}_{A}(p)]}{|\!|\sum_{A\in (\Cl^+_K)^2}  [\mathcal{C}_{A}(p)]|\!|_{B,\infty}}+v_{E}(p)\right|\! \right|_{B,\infty} \ll_{\eps} p^{2+\eps} d_K^{-1/12+\eps}. \end{equation}

\end{thm} 
As above we get the following corollary.
\begin{cor}
Let $p$ and $K$ be as above and consider a basic basis ${B}\subset V_p$ of level $p$ containing $v_{E}(p)$. Then for $d_K\gg_\eps p^{24+\eps}$, the $v_{E}(p)$-coordinate of the vector
$$\sum_{A\in (\Cl_K^+)^2} [\mathcal{C}_{A}(p)]\in V_p,$$ 
in the basis ${B}$ has strictly maximal absolute value among all coordinates (and in particular is non-zero).    
\end{cor}

The above result has a very close analogue in the imaginary quadratic case as worked out by Liu--Masri--Young \cite{LMY15} who studied a level $p$ version of the equidistribution of the map (\ref{eq:michelCM}). One can identify the finite set $\mathcal{E\ell\ell}^{ss}(\mathbb{F}_{p^2})=\{e_1,\ldots, e_n\}$ with the connected components of a certain conic curve $X^{p,\infty}$ defined from the quaternion algebra over $\Q$ ramified at $p$ and $\infty$. Thus $\{e_1,\ldots, e_n\}$ defines a basis for the $0$-th homology group $H_0(X^{p,\infty},\Z)$. In this language the results of \cite{LMY15} can be phrased exactly as (\ref{eq:thm1}) (see (\ref{eq:weakmichel2})). We will develop this analogy in greater detail in Section \ref{sec:CM}.   
\subsection{Applications}\label{sec:appl}
We will now present some applications of our results; one is group theoretic and the other has to do with non-vanishing of cycle integrals of modular forms. 

Consider a prime $p\equiv -1 \modulo 12$ (for simplicity) and put $n=\frac{p+1}{3}$. Then $\Gamma_0(p)$ considered as a subgroup of $\PSL_2(\Z)$ is torsion-free with $\Gamma_0(p)^\mathrm{ab}\cong \Z^{n/2+1}$ and thus we know by the Kurosh subgroup theorem that $\Gamma_0(p)$ is a free group on $n/2+1$ generators (being a subgroup of $\PSL_2(\Z)\cong \Z/2\Z \ast \Z/3\Z $). Let
$$\frac{0}{1}=\frac{a_{0}}{b_{0}}<\frac{a_1}{b_1}<\ldots<\frac{a_{n-1}}{b_{n-1}}< \frac{a_{n}}{b_{n}}=\frac{1}{1}, $$
be a \emph{Farey symbol of level $p$} in the terminology of Kulkarni \cite{Kulkarni91}, meaning that $a_i/b_i$ are reduced fractions such that $a_ib_{i+1}-a_{i+1}b_i=1$ for all $1\leq i<n$ and that there is a pairing $i\leftrightarrow i^\ast$ on $0,\ldots, n-1$ satisfying 
$$ b_ib_{i^\ast}+b_{i+1}b_{i^\ast+1}\equiv 0 \modulo p.  $$ 
Such a symbol always exists, even one that is symmetric around $1/2$, by \cite[Section 13]{Kulkarni91}. It follows from \cite{Kulkarni91} and a classical result of Poincar\'{e} that  $\Gamma_0(p)$ is freely generated by (the images inside $\PSL_2(\Z)$ of) $T=\begin{psmallmatrix} 1 & 1 \\ 0 & 1\end{psmallmatrix}$ together with the $n/2=\frac{p+1}{6}$ matrices  
\begin{align}\label{eq:basisK} \begin{pmatrix} a_{i^\ast+1}b_{i+1}+a_{i^\ast}b_i & -a_ia_{i^\ast}-a_{i+1}a_{i^\ast+1} \\ b_ib_{i^\ast}+b_{i+1}b_{i^\ast+1} & -a_{i+1}b_{i^\ast+1}-a_{i}b_{i^\ast}\end{pmatrix} \text{ for $i< i^\ast$  a pair}. \end{align}
The following group theoretic application can be thought of as an analogue of Linnik's Theorem on the smallest prime in arithmetic progressions (see also Theorem \ref{thm:LMYcor} below).

\begin{cor}\label{cor:group}
Let $p\equiv -1\modulo 12$ be prime and consider a real quadratic field $K$ of discriminant $d_K $ and (wide) class number one such that $p$ splits in $K$ with $p\mathcal{O}_K=\mathfrak{p}_1\mathfrak{p}_2$ such that $\mathfrak{p}_1$ does not have a generator of positive norm. Let $(u,v)$ be the positive half-integer solution to $u^2-d_Kv^2=1$ such that $v$  is minimal among all such and let $a,b,c\in \Z$ satisfy $b^2-4ac=d_K$ and $p|a$. Then for $d_K \gg_\eps p^{24+\eps}$ the matrix $\begin{psmallmatrix} u+bv & 2cv\\-2av & u-bv\end{psmallmatrix}\in \Gamma_0(p)$ is \underline{not} contained in the subgroup generated by the matrices (\ref{eq:basisK}). 
\end{cor}
The conjugacy class in $\Gamma_0(p)$ of the matrix $\begin{psmallmatrix} u+bv & 2cv\\-2av & u-bv\end{psmallmatrix}$ in the corollary above corresponds exactly to one of the two oriented closed geodesic associated to the class group of $K$ (see Section \ref{sec:geodesic} for details). We also obtain a related result when the wide class number of $K$ is not $1$ and for general prime levels $p$ which is however a bit more cumbersome to state. We will refer to Section \ref{sec:applproofs} for details.

The next applications is concerned with the non-vanishing of integrals of modular forms over closed geodesics. Let $\sigma_1(n)=\sum_{d| n} d$ be the sum of divisors function. For a Hecke eigenform $f\in \mathcal{M}_2(p)$ with Fourier coefficients $a_f(n)$ (at $\infty$), we have the trivial bound $|a_f(n)|\ll \sigma_1(n)$. This means that for any non-zero modular form $f\in \mathcal{M}_2(p)$ we can define;
$$M_{f}:=\inf \{c\geq 0: |a_f(n)|\leq c \sigma_1(n) ,\, \forall n\geq 1\}\in (0,\infty), $$  
where again $a_f(n)$ denotes the Fourier coefficients of $f$ at $\infty$.
\begin{cor}\label{cor:modular}
Let $p$ be prime and let $f\in \mathcal{M}_2(p)$ be a holomorphic modular form of weight $2$ and level $p$ with constant Fourier coefficient equal to $1$.
Consider a real quadratic field $K$ of discriminant $d_K $ and (wide) class number one such that $p$ splits in $K$ with $p\mathcal{O}_K=\mathfrak{p}_1\mathfrak{p}_2$ such that $\mathfrak{p}_1$ does not have a generator of positive norm. Let $C$ denote the geodesic associated to the class group of $K$. Then we have for $d_K\gg_\eps (M_{f})^{12+\eps} p^{48+\eps}$ that
$$\int_{C} f(z)dz\neq 0 .$$ 
Notice that the above  does not depend on the choice of orientation of $C$.
\end{cor}

\subsection{Geometric interpretation}\label{sec:geoint}
We will now explain the geometric content of Theorem \ref{thm:std3}. Recall that topologically $Y_0(p)$ is a genus $g$ curve with two punctures where the genus satisfies $g=\frac{p}{12}+O(1)$. We are interested in understanding the sum of the homology classes of the oriented closed geodesics of the principal genus, say. The associated geodesics will travel around $Y_0(p)$ in a complicated way but our results can be interpreted as saying that 
$$\textit{{\lq\lq}closed geodesics from the principal genus winds around the cusp at infinity a lot{\rq\rq}}$$ 

To illustrate this, let us consider the simplest non-trivial case $p=11$ where the genus is one. A fundamental polygon for $Y_0(11)$ is given by the hyperbolic polygon with vertices $\infty, 0, \tfrac{1}{3},\tfrac{1}{2},\tfrac{2}{3},1$ as illustrated in Figure \ref{fig:1}. The associated side pairing transformations (forgetting inverses) are 
$$T=\begin{psmallmatrix} 1 & 1 \\ 0 & 1 \end{psmallmatrix},\, A=\begin{psmallmatrix} 3 & -2 \\ 11 & -7 \end{psmallmatrix},\, B=\begin{psmallmatrix}4 & -3 \\ 11 & -8 \end{psmallmatrix},$$
which define a set of free generators for $\Gamma_0(11)$ (this follows from Poincar\'{e}'s Theorem as explained in Section \ref{sec:dirichlet}). As illustrated in Figure \ref{fig:2}, when viewing $Y_0(11)$ as a double punctured torus, the matrix $T$ corresponds to a simple loop around the puncture at $\infty$ and $A,B$ correspond to the loops going around the two {\lq\lq}holes{\rq\rq} of the torus. The content of Theorem \ref{thm:std3} is concerned with the coordinates of the closed geodesics in the basis $\{T,A,B\}$ in the homology (or equivalently when the corresponding hyperbolic conjugacy classes are projected to the abelinization of $\Gamma_0(11)$). Now consider the quadratic field $K=\Q(\sqrt{23})$ which has narrow class number $2$ and wide class number $1$. In this case we have associated to the principal class $I\in \Cl_K^+$, the conjugacy class  of the following matrix
$$\gamma_I:=\begin{psmallmatrix}26 & -35 \\ 55 & -74 \end{psmallmatrix},$$ 
and one can check that we have 
$$ \gamma_I=BA^{-1}T^{-1} B^{-1}T^{-1},$$
in our basis (see Figure \ref{fig:1}). Observe that $Y_0(11)$ is homotopic to a wedge of three circles and we are counting the (oriented) number of times the geodesic goes around each of the three circles as illustrated in Figure \ref{fig:2}.  We already see in this numerically very small example a tendency towards large $T$-coordinate.   
\begin{figure}
\begin{tikzpicture}[scale=11]
    \begin{scope}
      \clip (-0.6,-0.1) rectangle (0.6,0.6);
              \draw[thick] (0.40909,0) circle(0.435984);
          
         \draw[thick](0.40909-1,0) circle(0.435984); 
          \draw[thick] (1.17252139-0.5-0.2179841352,0) circle(0.2179841352); 
          \draw[thick] (1.17252139-0.5-1-0.2179841352,0) circle(0.2179841352); 
          \draw[thick] (0.718838463-0.5-0.03353053045,0) circle(0.03353053045);   
         
          \fill[white](-0.083333,0) circle(0.08333333);
      \fill[white](0.0833333,0) circle(0.083333333);
         \fill[white] (0.33333333,0) circle(0.16666666);
         \fill[white] (-0.3333333,0) circle(0.166666666);
         \draw (-0.083333,0) circle(0.08333333);
      \draw (0.0833333,0) circle(0.083333333);
         \draw (0.33333333,0) circle(0.16666666);
         \draw  (-0.3333333,0) circle(0.166666666);
          
         
     \fill[white](-0.5,0) rectangle  (-0.6,0.5);
        \fill[white] (0.5,0) rectangle  (0.6,0.5);
         \fill[white] (-0.6,-0.3) rectangle  (0.6,0);
         
      \draw (-0.5,0) -- (-0.5,0.5);
        \draw  (0.5,0) -- (0.5,0.5);
        
        \node at (0,0) [below] {{\tiny $ \tfrac{1}{2}$}};
         \node at (0.166666,0) [below] {{\tiny $ \tfrac{2}{3}$}};
          \node at (-0.166666,0) [below] {{\tiny $ \tfrac{1}{3}$}};
           \node at (0.5,0) [below] {{\tiny $1$}};
           \node at (-0.5,0) [below] {{\tiny $0$}};
           \node at (-0.0833333,0.0833333) [below] {{\tiny $B$}};
           \node at (0.0833333,0.0833333) [below] {{\tiny $A^{-1}$}};
           \node at (-0.33333,0.16666) [below] {{\tiny $A$}};
           \node at (0.33333,0.16666) [below] {{\tiny $B^{-1}$}};
           \node at (-0.5,0.3) [left] {{\tiny $T^{-1}$}};
           \node at (0.5,0.3) [right] {{\tiny $T$}};
    
    \end{scope}
    
      \end{tikzpicture}
  \caption{The closed geodesic on $Y_0(11)$ associated to the principal class in $\Q(\sqrt{23})$}\label{fig:1}
\end{figure}
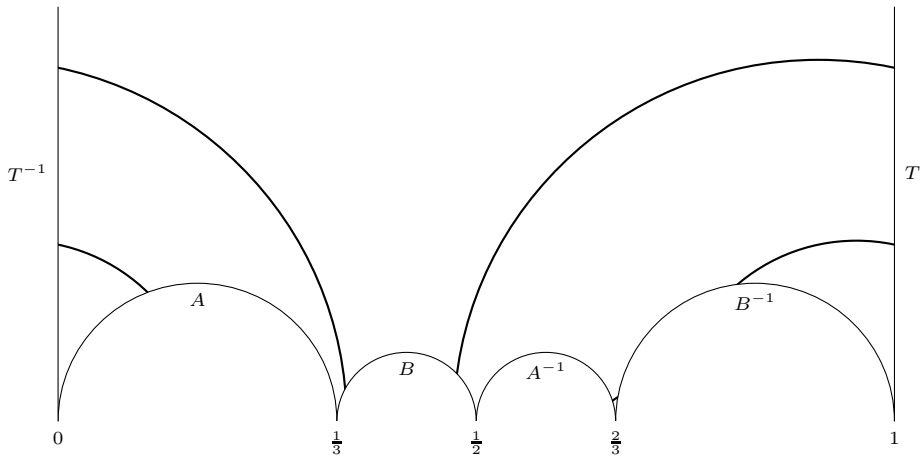
$\qquad$
\begin{figure}


\begin{tikzpicture}[yshift=1.5cm,scale=2] 
 
    \begin{scope}[yshift=11mm]\draw[yscale=cos(70), double distance=2*5mm] (0:1) arc (0:180:1);
    \draw[yscale=cos(70),  double distance=2*5mm] (180:1) arc (180:360:1);
     \fill[white](-12.5 mm,-0.07mm) rectangle  (-7 mm ,0.09mm);
     \fill[white](12.5 mm,-0.07mm) rectangle  (7 mm ,0.09mm);
 
    \draw [yscale=cos(70), densely dashed] circle(9mm);
    \draw [ xscale=cos(70), yshift=-3.4 mm, densely dashed]  circle(0.255);
    \draw [yscale=cos(55), dash pattern=on 1.5pt off 1.5pt, xshift=-11 mm] circle(0.1);
    \end{scope}
      \begin{scope}[yshift=-4mm]
     \node[xshift=-2*11 mm] at (0,15mm) {$\cdot$};
    \node[xshift=2*11 mm] at (0,15mm) {$\cdot$};
    \node[xshift=-2*10.5 mm, above] at (0,15.5mm) {{\tiny $T$}}; 
    \node[xshift=2*15 mm] at (0,15mm)  {$\simeq$}; 
    \node[yshift=1.33333*27 mm, above]  {{\tiny $A$}};
    \node[yshift=1.33333*16 mm, xshift=1.33333*0.6mm,  right]  {{\tiny $B$}}; 
   
    \end{scope}
    \node[xshift=1.5*15 mm]   {$$};  
    
   
  \end{tikzpicture}
  \begin{tikzpicture}[scale=0.8]
   \begin{polaraxis}[grid=none, axis lines=none]
     \addplot[mark=none,domain=0:360,samples=300] {cos(x*3)};
    \node at (0,0.8)  {$B$};
    \node at (120,0.8)  {$T$};
    \node at (240,0.8)  {$A$};
      \node at (340,0.8) [below] {{\color{black!60}$1-1=$}$\,0$};
      \node at (75,0.8) [below] {{\color{black!60}$0-2=$}$\,-2$};
      \node at (270,0.8) [below] {{\color{black!60}$0-1=$}$\,-1$};
   \end{polaraxis}
 \end{tikzpicture}
 
\caption{The $\{T,A,B\}$-coordinates in the homology of $Y_0(11)$ of the closed geodesic associated to the principal class of $\Q(\sqrt{23})$} \label{fig:2}
  \end{figure}
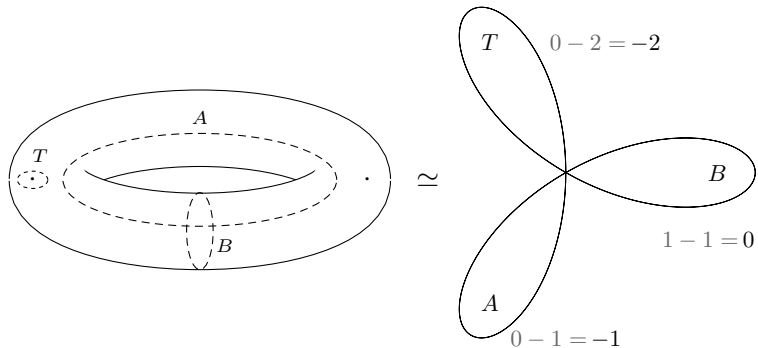
\begin{remark}\label{rem:risager}
Our results can be interpreted in the context of the classical problem of understanding the distribution of the projection 
\begin{equation}\label{eq:PhSa} \pi_1(M)\twoheadrightarrow \mathrm{Conj}(\pi_1(M))\twoheadrightarrow \pi_1(M)^\mathrm{ab}\cong H_1(M,\Z),\end{equation}
for a manifold $M$. Note that if $\Gamma_0(p)$ is torsion free then we have $\pi_1(Y_0(p))\cong \Gamma_0(p)$ and $H_1(Y_0(p), \Z)\cong \Gamma_0(p)^\mathrm{ab}$. For $M$ a compact Riemann surface there is a $1$-to-$1$ correspondence between conjugacy classes in $\pi_1(M)$ and oriented closed geodesics. 
Phillips and Sarnak \cite{PhillipsSarnak87} obtained an asymptotic expansion for the number of primitive geodesics of length $\leq X$ with specified image under the map (\ref{eq:PhSa}). In the same setup, Petridis and Risager \cite{PetridisRisager08} obtained an equidistribution statement for subsets $A\subset H_1(M,\Z)$ with asymptotic density. Also Petridis and Risager \cite{PetridisRisager05} showed that given a splitting $H_1(M,\Z)=\Z v \oplus V $ with $v\in H_1(M,\Z)$ then the $v$-coordinate of closed geodesics become normally distributed (when properly normalized) again ordered by the length of the geodesics. See also \cite{ConsNor20}, \cite{BurringEssen22}. 
From an arithmetic point of view the ordering by geodesic length is not very natural as it gives very large weight to discriminants with large class group (compared to the discriminant). Whereas the ordering by discriminant does not seem to admit any nice geometric description.
\end{remark}

\section*{Acknowledgements}
The author would like to thank Peter Humphries, Raphael Steiner, Daniel Kriz, and Farrell Brumley for useful discussions. The author would also like to thank the refer\'{e}e for a careful and insightful reading with many valuable suggestions as well as pointing out a mistake in a key step of the proof. The author's research was supported by the Independent Research Fund Denmark DFF-1025-00020B.

\section{Supersingular reduction of CM elliptic curves}\label{sec:CM} 
The result of Duke \cite{Du88} mentioned in the introduction regarding equidistribution of closed geodesics has an imaginary quadratic analogue, which amounts to the equidistribution of elliptic curves with complex multiplication inside the moduli space of elliptic curves over $\C$ (i.e. CM points on the modular curve, see again \cite{Du88}). Similarly, our results can be seen as a real quadratic analogue of the distribution of the supersingular reduction of elliptic curves with complex multiplication as investigated by many authors \cite{Michel04}, \cite{ElkiesOnoYang05}, \cite{Yang08}, \cite{Kane09}, \cite{LMY15}, \cite{AkaLuethiMichelWieser22}. This analogy between oriented closed geodesics in homology and supersingular reduction of CM elliptic curves is very natural and appears e.g in the recent work of Darmon--Harris--Rotger--Venkatesh \cite{DaHaRoVe21}.  

Let $\mathcal{E\ell\ell}^{ss}(\mathbb{F}_{p^2})=\{e_1,\ldots, e_n\}$ denote the set of isomorphism classes of supersingular elliptic curves defined over $\mathbb{F}_{p^2}$. It is known that $|\mathcal{E\ell\ell}^{ss}(\mathbb{F}_{p^2})|=\frac{p}{12}+O(1)$. We endow $\mathcal{E\ell\ell}^{ss}(\mathbb{F}_{p^2})$ with the measure defined by
\begin{equation}\label{eq:measure}\mu_p(\{e_i\})= w_i^{-1}/\sum_{j=1}^n w_j^{-1},\end{equation}
where $w_i$ denotes the size of the endomorphism group of the elliptic curves corresponding to $e_i$. Let $\mathcal{E\ell\ell}_K$ denote the set of isomorphism classes of elliptic curve (defined over $\overline{\Q}$) with complex multiplication by the ring of integers of the imaginary quadratic field $K$ with discriminant $d_K <0$ and class group $\Cl_K$. The set $\mathcal{E\ell\ell}_K$ carries a natural $\Cl_K$-action which is free and transitive. If $p$ is inert in $K$ we have a map 
$$r_\mathfrak{p}:\mathcal{E\ell\ell}_K\rightarrow \mathcal{E\ell\ell}^{ss}(\mathbb{F}_{p^2})$$ 
given by taking the mod $\mathfrak{p}$ reduction where $\mathfrak{p}$ is a prime ideal of the Hilbert class field of $K$ lying over $p$. The following is \cite[Theorem 3]{Michel04}. 

\begin{thm}[Michel] \label{thm:michel}Consider a CM elliptic curve $E\in \mathcal{E\ell\ell}_K$ and a subgroup $H\leq \Cl_K$ of index $\leq |d_K|^{1/2015}$. Then the orbits $H.E=\{r_\mathfrak{p}(A.E)\in  \mathcal{E\ell\ell}^{ss}(\mathbb{F}_{p^2}):A \in H \}$ become equidistributed as $d_K \rightarrow -\infty$ with respect to the measure $\mu_p$ given by (\ref{eq:measure}). \end{thm}

The proof goes through an identification of $\mathcal{E\ell\ell}^{ss}(\mathbb{F}_{p^2})$ with the set of connected components of a certain conic curve; 
\begin{equation}\label{eq:conic} X^{p,\infty}:=\mathbf{PB}^\times(\Q) \backslash \mathbf{PB}^\times(\A_\Q) /\mathbf{PB}^\times(\widehat{\Z}) K_\infty,\end{equation}
where $\mathbf{B}$ denotes the unique quaternion algebra over $\Q$ ramified at $p$ and $\infty$, $\mathbf{PB}^\times$ its group of projective units, $K_\infty$ denotes a maximal compact torus of $\mathbf{PB}^\times(\R)$ and $\mathbf{PB}^\times(\widehat{\Z})$ denotes the projective units of $\widehat{\Z}\otimes \mathcal{O}$ where $\mathcal{O}\subset \mathbf{B}$ is a maximal order. Using the Jacquet--Langlands correspondence and a formula of Gross, this reduces the distribution problem to subconvexity bounds of certain Rankin--Selberg $L$-functions which is resolved (see \cite[Section 5]{Michel04} for details). 

In order to set up the analogy with the real quadratic case, we will slightly reformulate the statement in Theorem \ref{thm:michel} above. Denote by $H_0(X^{p,\infty},\Z)$ the $0$-th (singular) homology group of $X^{p,\infty}$ with integral coefficients (one should picture a copy of $\Z$ at each connected component of $X^{p,\infty}$), which is a lattice inside the real homology group $H_0(X^{p,\infty},\R)$. Note that both of these abelian groups carry a natural action of the Hecke algebra coming from the description in terms of quaternion algebras (see e.g.\cite[Section 1.5]{BertoDarmon96}), and as such is isomorphic to the space of modular forms $\mathcal{M}_2(p)$ of level $p$ and weight $2$. We have a natural basis of $H_0(X^{p,\infty},\Z)$ of geometric nature corresponding to the classes $e_1,\ldots, e_n$ (using suggestive notation) associated to each connected component of $X^{p,\infty}$. In this basis the Eisenstein element is the following
$$ e_0:= \sum_{i=1}^n w_i^{-1}e_i, $$
meaning that $T_\ell e_0=(\ell+1)e_0$ for $\ell\neq p$ prime and $T_\ell$ the $\ell$-th Hecke operator. Using the above identifications with supersingular elliptic curves and after fixing an elliptic curve $E\in \mathcal{E\ell\ell}_K$, we get a map
\begin{equation}\label{eq:mapssEC}r_p: \Cl_K\rightarrow \{e_1,\ldots, e_n\}\subset H_0(X^{p,\infty},\Z)\subset H_0(X^{p,\infty},\R),\quad A\mapsto r_\mathfrak{p}(A.E),\end{equation}
which will serve as an imaginary quadratic analogue of the map (\ref{eq:maphom}). We will now consider Theorem \ref{thm:michel} as a statement about convergence (with respect to the standard topology) on the $(n-1)$-sphere which we identity with $\mathbf{S}(V_{p,\infty}):=(V_{p,\infty}-\{0\})/\R_{>0}$ equipped with the quotient toplogy where $V_{p,\infty}:=H_0(X^{p,\infty},\R)\cong \R^{n}$ (equipped with the Euclidean topology). As above for $v\in V_{p,\infty}-\{0\}$ we denote by $\overline{v}\in \mathbf{S}(V_{p,\infty})$ the image under the natural projection $V_{p,\infty}-\{0\}\twoheadrightarrow \mathbf{S}(V_{p,\infty})$. We can then recast the equidistribution statement of Michel as follows. 


\begin{thm}[{\lq\lq}Vector space{\rq\rq}-version of Theorem \ref{thm:michel}]\label{thm:michel2} Let $p>2$ be prime and let $K$ be an imaginary quadratic field of discriminant $d_K <0$ such that $p$ is inert in $K$. Consider a subgroup $H\leq \Cl_K$ of index $\leq |d_K|^{1/2015}$ and a coset $CH\subset \Cl_K$. Then we have as $d_K \rightarrow -\infty$ that  
\begin{equation}\label{eq:weakconvm}\overline{\sum_{A \in CH} r_p(A)} \longrightarrow \overline{e_0}, \end{equation}
in the standard topology of $\mathbf{S}(V_{p,\infty})$.
\end{thm}



We will now show that this is equivalent to Theorem \ref{thm:michel}. Let $B=\{e_1,\ldots, e_n\}$ be the standard basis for $V_{p,\infty}$ which defines an isomorphism $V_{p,\infty}\cong \R^n$. As explained above, we get a norm $|\!|\cdot |\!|_{B,1}$ on $V_{p,\infty}$ by pulling back the $L^1$-norm in the standard basis of $\R^n$. Now recall that the convergence on the $(n-1)$-sphere $\mathbf{S}(V_{p,\infty})$ is equivalent to the convergence statement (\ref{eq:convergenceformulation}) using e.g. the norm $|\!|\cdot |\!|_{B,1}$.  Notice that we have
$$ |\!|\sum_{A \in CH} r_p(A) |\!|_{B,1}=|H|,\quad   |\!|e_0 |\!|_{B,1}=\sum_{i=1}^n w_i^{-1},$$
recalling that the $e_i$-coordinate of $e_0$ is equal to $w_i^{-1}$. Now the convergence in (\ref{eq:convergenceformulation}) implies that the $e_i$-coordinates of $\sum_{A \in CH} r_p(A)$ converge to those of $e_0$ (both normalized). This recovers the statement of Michel as in Theorem \ref{thm:michel} up to the fact that Theorem \ref{thm:michel2} does not see that $r_p(A)$ is always equal to one of the vectors $e_1,\ldots, e_n$.


\subsubsection{The case of varying level} It is natural to ask what happens if we let $p$ vary with $d_K $ as has been considered by Liu--Masri--Young \cite{LMY15}. In the above terminology they consider the convergence with respect to the basis corresponding to the connected components $\{e_1,\ldots, e_n\}$ of $X^{p,\infty}$. Then \cite[Theorem 1.1]{LMY15} amounts to the following;
\begin{equation}\label{eq:weakmichel2} 
\left|\! \left|\frac{\sum_{A \in \Cl_K} r_p(A)}{|\!|\sum_{A \in \Cl_K}  r_p(A)|\!|_{B,1}}-\frac{e_0}{|\!|e_0|\!|_{B,1}}\right| \!\right|_{B,\infty} \ll_\eps p^{1/8+\eps} |d_K|^{-1/16+\eps}, \end{equation}
where again $|\!| \cdot |\!|_{B,1}$ and $|\!| \cdot |\!|_{B,\infty}$ denote the pullback of, respectively the $L^1$-norm and sup norm under the isomorphism $V_{p,\infty}\cong \R^n$ defined by  $B=\{e_1,\ldots, e_n\}$. We note that the statement (\ref{eq:weakmichel2}) is exactly parallel to Theorem \ref{thm:main}.
Translating back to the language of elliptic curves, (\ref{eq:weakmichel2}) implies the following analogue of Linnik's Theorem on the smallest prime in arithmetic progressions.
\begin{thm}[Liu--Masri--Young]\label{thm:LMYcor}
The reduction map $r_\mathfrak{q} : \mathcal{E\ell\ell}_K \rightarrow \mathcal{E\ell\ell}^{ss}_{p^2}$
is surjective for $|d_K| \gg_\eps p^{18+\eps}$.
\end{thm}
We think of Corollary \ref{cor:group} (and more generally Corollary \ref{cor:groupgen} below) as a real quadratic analogue of the above.

\section{Arithmetic background}
In this section we will introduce some basic facts about respectively, oriented closed geodesics associated to class groups of real quadratic fields and (co)homology of modular curves. 
 \subsection{Real quadratic fields and closed geodesics}\label{sec:geodesic}
We will refer to \cite{Popa06} and \cite{Darmon94} for an in-depth account of the following material. Let $K$ be a real quadratic extension of $\Q$ of discriminant $d_K >0$. Let $p$ be a prime which splits in $K$ and fix throughout a residue $r \modulo 2p$ such that $r^2\equiv d_K\modulo 4p$. Then we have the following equality of ideals
$$ p\mathcal{O}_K=  \left[p,\frac{r-\sqrt{d_K}}{2}\right]\left[p,\frac{r+\sqrt{d_K}}{2}\right],$$
where we use the following notation for $\alpha,\beta\in K$;
\begin{equation}\label{eq:alphabeta}[\alpha,\beta]:=\Z\alpha+\Z\beta \subset K.\end{equation}
In other words $[p,\frac{r\pm \sqrt{d_K}}{2}]$ are the two prime ideals of $\mathcal{O}_K$ lying over $p$.  

We let $\Cl_K^+$ denote the narrow class group of $K$ (i.e. fractional ideals modulo principal ideals generated by elements of positive norms). Below when the discriminant $d_K $ is clear from context, we let $I\in \Cl_K^+$ denote the class containing the principal fractional ideal $(1)=\mathcal{O}_K\subset K $ and $J\in \Cl_K^+$ denote the class containing the different $(\sqrt{d_K})=\sqrt{d_K}\mathcal{O}_K\subset K $. Observe that for fundamental discriminants divisible by a prime $q\equiv 3\modulo 4$ there exists no unit with norm $-1$ (as $-1$ is not a quadratic residue modulo $q$) which implies that $J\neq I$. In this case, a principal ideal belongs to the class $J$ exactly if it has a generator with negative norm.

We will now show the following result, which is closely related to the conditions appearing in our main results as we will see below. 

\begin{prop}\label{prop:H}
Let $p$ be an odd prime. Then there exists infinitely many positive fundamental discriminants $d$ such that $J\neq I$ and $[p,\frac{r-\sqrt{d}}{2}]\in J$ inside $\Cl_K^+ $ with $K=\Q(\sqrt{d})$.
\end{prop}
\begin{proof}Pick an odd prime $q\equiv 3\modulo 4$ such that $-p$ is a quadratic residue mod $q$. Then by simple considerations about quadratic residues there is a residue $t\modulo 8 q$ such that 
$$ q| t^2+p,\, \, q^2\!\!\not| t^2+p \quad \text{and}\quad (t^2+p \modulo 16)\in \{ 1, 5, 8,9,12,13\}, $$
meaning in particular that $t^2+p\equiv 0\text{ or }1\modulo 4$. Now for each $n\equiv t \modulo 8q$ there is a unique positive fundamental discriminant $d >0$  and integer $m>0$ such that $d m^2=n^2+p$. By construction we have $q| d$ and thus  $J\neq I$ inside $\Cl_K^+$ where $K=\Q(\sqrt{d})$. Furthermore, we have the factorization 
$$-p=(n-\sqrt{d}m)(n+\sqrt{d}m),$$ 
which means that $[p,\frac{r-\sqrt{d}}{2}]$ is a principal ideal generated by $ n\pm \sqrt{d}m$ (for some choice of sign $\pm$). By construction the norm of $n\pm \sqrt{d}m$ is negative (and equal to $-p$). By the above this implies $[p,\frac{r-\sqrt{d}}{2}]\in J$ in $\Cl_K^+$ as wanted.
\end{proof} 
Now it follows that if $J\neq I$ then $J\notin (\Cl_K^+)^2$. 
Thus for $p,d$ as in Proposition \ref{prop:H} and each subgroup $H\leq (\Cl_K^+)^2$, we have $[p,\frac{r-\sqrt{d}}{2}]\notin H$. This gives plenty of examples for which Theorems \ref{thm:std3} and \ref{thm:main} apply, recalling that $\{\mathfrak{p}_1, \mathfrak{p}_2\}=\{[p,\frac{r\pm\sqrt{d}}{2}]\} $.   

\subsubsection{Closed geodesics and class groups}Let $p$ be prime and $K$ a real quadratic field of discriminant $d_K$ such that $p$ splits in $K$. Let $r\modulo 2p$ as above be such that $r^2\equiv d_K\modulo 4p$. Denote by $\mathcal{Q}_{K,p}$ (suppressing $r$ in the notation) the set of all integral binary quadratic form 
$$Q(x,y)=ax^2+bxy+cy^2$$
of discriminant $b^2-4ac=d_K$ and level $p$ meaning that $a\equiv 0\modulo p$ and $b\equiv r \modulo 2p$. The group $\Gamma_0(p)$ acts naturally on $\mathcal{Q}_{K,p}$ by coordinate transformation. It is a classical fact essentially due to Gau{\ss} that there is a natural bijection (depending on the choice of $r \modulo p$)
\begin{equation}\label{eq:Gaussiso} \Cl_K^+ \xrightarrow{\sim} \Gamma_0(p)\backslash \mathcal{Q}_{K,p}. \end{equation}
When the level is trivial the above bijection is induced by mapping $ax^2+bxy+cy^2$ to the narrow ideal class of the fractional ideal $[1,\frac{-\mathrm{sign}(a)b+\sqrt{d_K}}{2|a|}]$ (using the notation (\ref{eq:alphabeta})). 

Given an integral binary quadratic form $Q(x,y)=ax^2+bxy+cy^2$ of discriminant $d_K $ and level $p$, we associate the following matrix
\begin{align}\label{eq:gammaQ}\gamma_Q:=\begin{pmatrix} u+bv & 2cv \\ -2av & u-bv\end{pmatrix}\in \Gamma_0(p),\end{align}
where $u,v$ are positive half-integers satisfying Pell's equation $u^2-d_Kv^2=1$ and such that $v$ is minimal among all such solutions (i.e. the fundamental positive unit of $K$ is $\epsilon_K=u+v\sqrt{d_K}$). 

For $A\in \Cl_K^+$, we denote by $\mathcal{C}_{A}(p)$ the oriented closed geodesic on $Y_0(p)$ obtained by projecting the oriented geodesic connecting $z_Q$ and $\gamma_Q z_Q$ where 
$$z_Q:=\frac{-\mathrm{sign} (a)b+i \sqrt{d_K}}{2|a|},$$
and $Q\in \mathcal{Q}_{K,p}$ corresponds to $A$ under the isomorphism (\ref{eq:Gaussiso}) (the image in $Y_0(p)$ is independent of the choice of representative $Q$). 

  \subsection{(Co)homology of modular curves}\label{sec:cohom}
Here and throughout we will consider matrices in $\SL_2(\R)$ as  elements of $\PSL_2(\R)$ without further mentioning. Let $p$ be prime  and consider the Hecke congruence group (or more precisely its projection to $\PSL_2(\R)$) 
$$\Gamma_0(p):=\{\begin{psmallmatrix} a&b\\c&d\end{psmallmatrix} \in \PSL_2(\Z): p|c\}.$$
Let 
$$Y_0(p):=\Gamma_0(p)\backslash \Hb,\quad  X_0(p):=\overline{Y_0(p)}=Y_0(p)\cup \Gamma_0(p)\backslash \mathbb{P}^1(\Q) $$ 
be resp., the modular curve of level $p$ and its compactification which is a compact Riemann surface of genus $g=\frac{p}{12}+O(1)$ (see e.g. \cite[Proposition 1.40]{Sh94}). We can consider the integral singular homology group \cite[Chapter 2]{Hatcher02} 
$$H_1(Y_0(p),\Z)\cong \Z^{2g+1},$$ 
which sits as a lattice inside the real homology 
$$H_1(Y_0(p),\R)\cong \R^{2g+1}.$$ 
We will be interested in the distribution of oriented closed geodesics inside the lattice $H_1(Y_0(p),\Z)$.  

We have the cap product pairing
\begin{equation}\label{eq:cap} \langle\cdot, \cdot \rangle: H_1(Y_0(p),\R)\times H^1(Y_0(p),\R)\rightarrow \R,\end{equation}
between real homology and cohomology which identifies $H^1(Y_0(p),\R)$ with the linear dual $H_1(Y_0(p),\R)^\ast$. Given a basis $B$ of $H_1(Y_0(p), \R)$ we denote by $ B^{\ast}\subset H^1(Y_0(p), \R)$ the dual basis of $B$ with respect to the cap product pairing as in (\ref{eq:cap}). 

Recall that the de Rham isomorphism gives a  description of $H^1(Y_0(p),\R)$ in terms of real valued harmonic 1-forms on $Y_0(p)$. Any such 1-form is a linear combination of forms of the type;
$$  \omega_f=2\pi i f(z)dz ,\quad  \overline{\omega_f}=\overline{2\pi i f(z)dz}, $$
where $f\in \mathcal{M}_2(p)$ is a weight 2 and level $p$ holomorphic form (not necessarily cuspidal). We also have a surjective map 
\begin{align}\label{eq:projab} \Gamma_0(p)\twoheadrightarrow H_1(Y_0(p), \Z),\quad \gamma\mapsto \{ z, \gamma z\}\end{align}
where $z\in \Hb$ (the class does not depend on the choice of $z$) and we are using the following notation for $z_1,z_2\in \Hb$;
\begin{equation}\label{eq:def}\{z_1, z_2\}:= [\text{class of the oriented geodesic connecting $z_1$ and $z_2$}]\in H_1(Y_0(p), \R), \end{equation}
which defines an element of the real homology $H_1(Y_0(p), \R)$ via integration against $1$-forms. The map (\ref{eq:projab}) induces an isomorphism
\begin{align}\Gamma_0(p)^\mathrm{ab}/(\Gamma_0(p)^\mathrm{ab})_\mathrm{tor}\cong H_1(Y_0(p),\Z).\end{align}
We note that for a closed oriented geodesic $\mathcal{C}_{A}(p)$ as in the previous section the homology class $[\mathcal{C}_{A}(p)]\in H_1(Y_0(p),\Z)$ corresponds exactly to the image of (any) $\gamma_Q$ as in (\ref{eq:gammaQ}) under the map (\ref{eq:projab}). Using the above identifications  the cap product pairing is induced by the map 
$$ \Gamma_0(p)\times \mathcal{M}_2(p)\ni (\gamma, f)\mapsto \int^{\gamma z}_z \omega_f,  $$
for any $z\in \Hb$. 


The natural pullback map induced by inclusion, fits into a short exact sequence of $\R$-vector spaces
\begin{equation}\label{eq:SES} 0\rightarrow H^1(X_0(p),\R)\rightarrow H^1(Y_0(p),\R)\rightarrow \R \rightarrow 0, \end{equation}
using here that we have two cusps. This identifies $H^1(X_0(p),\R)$ with the parabolic classes in $H^1(Y_0(p),\R)$ (i.e. classes which vanish on all parabolic elements of $\Gamma_0(p)$ using the above pairing). Under the de Rham isomorphism the parabolic classes correspond to $1$-forms obtained from holomorphic cusp(!) forms of weight $2$ and level $p$ (see e.g. \cite[Section 5]{Sh94}). Similarly, we have the pushforward map $H_1(Y_0(p),\R)\rightarrow H_1(X_0(p),\R)$ whose kernel is exactly given by the image of the parabolic elements of $\Gamma_0(p)\otimes \R$ inside $H_1(Y_0(p),\R)$ under the map (\ref{eq:projab}).


\subsubsection{Hecke operators}\label{sec:cohomH}Due to the arithmetic nature of $\Gamma_0(p)$ we have a family of commuting linear operators acting on all of the above mentioned (co)homology groups, namely the \emph{Hecke operators}. The Hecke action is induced by the following; on the space of holomorphic forms $\mathcal{M}_2(p)$ of weight 2 and level $p$ the $n$-th Hecke operator $T_n$ acts by (see e.g.\cite[(14.46)]{IwKo})
\begin{equation} T_n f(z):= \frac{1}{n}\sum_{\substack{ad=n,\\ (a,p)=1}} a^2 \sum_{0\leq b<d} f\left( \frac{az+b}{d} \right), \end{equation} 
the Fricke involution $W_p$ acts as 
\begin{equation}\label{eq:Fricke} W_p f(z):=p^{-1}z^{-2} f(-1/(pz)),\end{equation}       
and we also have the involution $\iota$ given by 
\begin{equation}\iota f(z):= f(-\overline{z}),\end{equation}
defined for $f\in \mathcal{M}_2(p)\oplus \overline{\mathcal{M}_2(p)}$.  Notice that with this normalization the Ramanujan conjecture amounts to the bound $\leq d(n)n^{1/2}$ for the Hecke eigenvalues. See also \cite[Section 5]{Sh94} for an intrinsic definition in terms of group cohomology using double cosets. Similarly we can define an action on the homology groups by (using the notation (\ref{eq:def}))
\begin{equation}\label{eq:Tn} T_n \{z_1, z_2\}:= \sum_{\substack{ad=n,\\ (a,p)=1}} \,\,\sum_{0\leq b<d} \left\{\frac{az_1 +b}{d}, \frac{a z_2 +b}{d}\right\},   \end{equation}  
\begin{equation}\label{eq:Frickehom}W_p\{z_1, z_2\}:=\{-1/(pz_1), -1/(pz_2)\}, \end{equation}        
and 
\begin{equation} \label{eq:iota}\iota \{z_1, z_2\}:= \{-\overline{z_1}, -\overline{z_2}\}.  \end{equation}
One can now check that all of these operators are pairwise adjoint with respect to the cap product pairing as above. Furthermore, it can be shown that all of these linear operators commute and thus we can find a common eigen-basis. Explicitly, such a Hecke eigen-basis for $H^1(Y_0(p),\R)$ is given by 
\begin{align}\label{eq:Heckebasis} B_\mathrm{Hecke}(p):=\{\omega_f^{\epsilon}: f\in \mathcal{B}_p,\epsilon\in\{\pm\}\}\cup\{ \omega_{E}(p) \}, \end{align}
where 
$$ \omega_f^{\pm}= \left( 2\pi i f(z)dz\pm \overline{2\pi i f(z)dz}\right)/(1+i \pm \overline{1+i}))\in H^1(Y_0(p), \R),$$
with $\mathcal{B}_p:=\{f_1,\ldots, f_g\}\subset \mathcal{S}_2(p)$ a basis of Hecke normalized (i.e. the first Fourier coefficient is $1$) holomorphic cuspidal eigen forms of weight $2$ and level $p$. And 
\begin{equation}\label{eq:eisensteinclass} \omega_{E}(p):=E_{2,p}(z) dz\in H^1(Y_0(p), \R),\end{equation}
is the normalized \emph{Eisenstein class} defined from the weight $2$ Eisenstein series of level $p$; 
$$ E_{2,p}(z):=\frac{pE_2(pz)-E_2(z)}{p-1}=\frac{pE^\ast_2(pz)-E^\ast_2(z)}{p-1}=1+\frac{24}{p-1}e^{2\pi i z}+\ldots,$$ 
where $E_2$ denotes the the weight $2$ Eisenstein series of level $1$ and $E^\ast_2$ the modified Eisenstein series given by
\begin{equation}\label{eq:E2}E_2(z):=1-24\sum_{n\geq 1} \sigma_1(n)e^{2\pi i n z}=E_2^\ast(z)+\frac{3}{\pi y}.\end{equation} 
Here $\sigma_1(n)=\sum_{d|n}d$ is the sum of divisor function. One obtains a Hecke eigen basis for the homology as the dual basis of $B_\mathrm{Hecke}(p)$ with respect to the cap product pairing (\ref{eq:cap}) which we denote by;
\begin{align}\label{eq:Heckebasisdual} \check{B}_\mathrm{Hecke}(p):=\{v_f^{\pm}: f\in \mathcal{B}_p,\pm\}\cup\{ v_{E}(p) \}\subset H_1(Y_0(p),\R) , \end{align}
where $\langle v_f^\pm,\omega_f^\pm\rangle=1$ for $f\in \mathcal{B}_p$ and $\langle v_E(p),\omega_E(p)\rangle=1$. Since the constant Fourier coefficient of $E_{2,p}$ is $1$ one sees that in fact  
\begin{equation}\label{eq:eisensteinclasshom}v_{E}(p)=\{i,i+1\}\in H_1(Y_0(p),\Z), \end{equation}
with notation as in (\ref{eq:def}). In other words $v_{E}(p)$ is the (normalized) Eisenstein class in homology appearing in Theorem \ref{thm:std3} which geometrically is the class of a simple loop going around the cusp at $\infty$. 


We have the following classical formula for $\gamma\in \PSL(\Z)$ (see e.g. \cite[(55)]{DuImTo18}):
$$ \int_{z}^{\gamma z} E_2^\ast(z) dz= \Psi(\gamma),$$
where 
$$\Psi(\begin{psmallmatrix} a &b\\c &d \end{psmallmatrix}):= \begin{cases} \frac{a+d}{c}-12\,\mathrm{sign}(c)s(a,c)-3 \,\mathrm{sign}(c(a+d)),& c\neq 0\\ \frac{b}{d},& c=0,\end{cases}$$
is the Rademacher symbol with 
$$s(a,c):=\sum_{n=1}^{ c } (\!(\tfrac{n}{c})\!)(\!(\tfrac{na}{c})\!),$$
the classical Dedekind sum with 
$$(\!(x)\!)=\begin{cases}x-\lfloor x\rfloor-1/2,& x\notin \Z,\\ 0,& x\in \Z.\end{cases}$$
the sawtooth function. Notice that $\Psi(\gamma)$ does not depend on the representative of $\gamma\in \PSL_2(\Z)$ as should be the case. This implies the following key formula for $\gamma\in \Gamma_0(p)$
\begin{equation}\label{eq:DedekindEisenstein}\langle \{z,\gamma z\}, \omega_{E}(p) \rangle= \frac{\Psi\left(\gamma'\right)- \Psi(\gamma)}{p-1},\end{equation}  
where 
$$\gamma'=\begin{psmallmatrix} p & 0 \\0 &1 \end{psmallmatrix}\gamma \begin{psmallmatrix} 1/p & 0 \\0 &1 \end{psmallmatrix}=\begin{psmallmatrix} a & bp \\c/p &d \end{psmallmatrix}\in \PSL_2(\Z), $$
with $\gamma=\begin{psmallmatrix} a & b \\c &d \end{psmallmatrix}$. Using the trivial bound $\Psi(\gamma)\ll \frac{|a+d|}{ c }+ c $ (for $c> 0$) we get the following useful estimate
\begin{equation}\label{eq:BoundEisenstein}\langle \{z,\gamma z\}, \omega_{E}(p) \rangle\ll  \frac{p|a+d|+ c ^2}{p c }, \quad \gamma=\begin{psmallmatrix} a & b \\c &d \end{psmallmatrix}\in \Gamma_0(p), c> 0.\end{equation}

Recall the short exact sequence (\ref{eq:SES}) above. A Hecke-equivariant splitting is given by mapping $1\in\R$ to the Eisenstein class $\omega_{E}(p)=E_{2,p}(z)dz$ (this fits into a general framework due to Franke \cite{Franke98}). This defines an isomorphism
\begin{equation}\label{eq:franke}H^1(Y_0(p),\R)\cong H^1(X_0(p),\R)\oplus \R \omega_{E}(p),\end{equation}
and we will denote the projection to the {\lq\lq}cuspidal subspace{\rq\rq} (which we will identify with $H^1(X_0(p),\R)$) with respect to this splitting by 
\begin{equation}\label{eq:franke1} \mathbb{P}_\mathrm{cusp}:H^1(Y_0(p),\R)\rightarrow H^1(X_0(p),\R)\subset  H^1(Y_0(p),\R), \end{equation}  
which is explicitly given by 
$$ \mathbb{P}_\mathrm{cusp} \omega= \omega- \langle v_{E}(p),\omega\rangle\omega_{E}(p),  $$
where $v_{E}(p), \omega_{E}(p)$ are the Eisenstein classes defined above.

\section{Background on Fuchsian groups} \label{sec:dirichlet}
In this section we will review some useful facts regarding the geometry of fundamental polygons of Fuchsian groups. We will refer to \cite[Section 9]{Beardon83} for an in-depth treatment.
\subsection{Fundamental polygons}\label{sec:fundpoly}
Let $\Gamma\subset \PSL_2(\R)$ be a Fuchsian subgroup of the first kind, i.e. a discrete and cofinite subgroup of $\PSL_2(\R)$. A  \emph{fundamental domain} for $\Gamma$ is a locally finite domain $\mathcal{F}\subset \Hb$ such that 
\begin{enumerate}
\item For any $z\in \Hb$, we have $\gamma z\in \overline{\mathcal{F}}$ for some $\gamma\in\Gamma$.
\item\label{item:2} If $z_1,z_2\in \overline{\mathcal{F}}$ are $\Gamma$-equivalent then $z_1=z_2$ or $z_1,z_2\in \partial \mathcal{F}$.
\end{enumerate}
 We say that $\mathcal{F}$ is a \emph{fundamental polygon} for $\Gamma$ if $\mathcal{F}$ is furthermore (hyperbolically) convex with piecewise geodesic boundary. We define a \emph{side of} $\mathcal{F}$ as a non-empty subset of the shape $\gamma \overline{\mathcal{F}}\cap \overline{\mathcal{F}}$ with $ \mathrm{Id}\neq \gamma \in\Gamma$. We define a \emph{vertex of} $\mathcal{F}$ as a non-empty subset of the shape $\gamma_1 \overline{\mathcal{F}}\cap \gamma_2 \overline{\mathcal{F}} \cap \overline{\mathcal{F}}$ with $ \mathrm{Id},\gamma_1,\gamma_2\in \Gamma$ pairwise distinct. It can be shown that a fundamental polygon $\mathcal{F}$ has an even  number of sides which are pairwise $\Gamma$-equivalent, and $\Gamma \mathcal{F}$ gives a tessellation of $\Hb$. The set of elements identifying sides of $\mathcal{F}$ are called the \emph{side pairing transformations} associated to $\mathcal{F}$ which we denote by $\mathcal{S}(\mathcal{F})\subset \Gamma$ (which formally is a multiset if there are order two elements in $\Gamma$). Given a side $L$ of $\mathcal{F}$ we refer to \emph{the side pairing transformation  associated to the side $L$} as the element $\sigma \in \mathcal{S}(\mathcal{F})$  such that $\sigma^{-1}L$ is also a side of $\mathcal{F}$. Similarly we say that $\sigma$ is the side pairing transformation associated to the side $\gamma L$ of the translate $\gamma \mathcal{F}$ for each $\gamma \in \Gamma$.  
  

It is known that $\mathcal{S}(\mathcal{F})$ generates $\Gamma$ for any fundamental polygon $\mathcal{F}$, and it is a fundamental fact that one can understand the relation between the elements of $\mathcal{S}(\mathcal{F})$ from the geometry of a fundamental polygons and visa versa (see Lemma \ref{lem:poincare}). A simple but key incarnation is the following.

\begin{prop} \label{prop:geocoding}
Let $\mathcal{F}$ be a fundamental polygon for a discrete and co-finite subgroup $\Gamma\subset \PSL_2(\R)$. Consider a sequence of consecutive $\Gamma$-translates of $\mathcal{F}$;
$$\mathcal{F},\gamma_1 \mathcal{F},\ldots, \gamma_n \mathcal{F}.$$ 
Then 
$$\gamma_n=\sigma_1\sigma_2\cdots \sigma_n,$$ 
where $\sigma_i\in \mathcal{S}(\mathcal{F})$ denotes the side pairing transformation associated to the side shared between $\gamma_{i-1} \mathcal{F}$ and $\gamma_{i} \mathcal{F}$ (here we put $\gamma_0=\mathrm{Id}$).
\end{prop} 
\begin{proof}Note that the side of $\mathcal{F}$ shared with $\gamma_1 \mathcal{F}$ will exactly have associated $\sigma_1=\gamma_1\in \mathcal{S}(\mathcal{F})$. Now if the side shared between $\gamma_1 \mathcal{F}$ and $\gamma_2 \mathcal{F}$ has associated $\sigma_2\in \mathcal{S}(\mathcal{F})$ then we have
$$\gamma_2\mathcal{F}=(\sigma_1\sigma_2\sigma_1^{-1})(\sigma_1 \mathcal{F})=\sigma_1 \sigma_2 \mathcal{F}.$$
Continuing like this we get 
$$\gamma_n\mathcal{F}=\sigma_1 \cdots \sigma_n \mathcal{F},$$
and since $\gamma \mathcal{F}=\gamma' \mathcal{F}$ implies $\gamma=\gamma'\in \Gamma$ by condition (\ref{item:2}), we conclude the wanted equality.
 \end{proof}
 A first application of Proposition \ref{prop:geocoding} is the following slight reformulation. 
\begin{cor}\label{cor:geocoding}Let $c:[0,1]\rightarrow \Hb$ be a continuous, injective curve with $c(0)=z\in \mathcal{F}$ and $c(1)=\gamma z$ for some $\gamma\in \Gamma$. Assume that $c([0,1])$ does not intersect the set of $\Gamma$-translates of the vertices of $\mathcal{F}$. Then one can \emph{code} the element $\gamma$ in the following way; let $\sigma_1,\ldots, \sigma_n\in \mathcal{S}(\mathcal{F})$ be the (ordered) sequence of side pairing transformations associated to the intersection between the curve $c$ and $\Gamma$-translates of the sides of $\mathcal{F}$. Then one has $$\gamma=\sigma_1\cdots\sigma_n.$$
\end{cor} 
When $c([0,1])$ is a closed geodesic (when projected to $\Gamma\backslash \Hb$) this is known as \emph{geometric coding} of geodesics (see \cite{Katok96} for a nice treatment). 

As a second consequence we note that if the sequence of translates {\lq\lq}loops around{\rq\rq}, meaning that $\gamma_n=1$, then one obtains a relation between the side pairing transformations $S(\Gamma)$. This yields immediately that if two sides of $\mathcal{F}$ are paired then the associated elements in $\mathcal{S}(\mathcal{F})$ are inverses. We call these the \emph{inverse relations} of $\mathcal{F}$. 
We also get relations by looping around the vertices of $\mathcal{F}$ which we will now make precise. Observe that the embedding $\mathcal{F}\subset \Hb$ defines an orientation on the boundary $\partial \mathcal{F}$. Let $L_1$ be a side of $\mathcal{F}$ with left most (wrt. the orientation) vertex $v_1$ and let $\sigma_1\in \mathcal{S}(\mathcal{F})$ be the associated side pairing transformation. Let $L_2$ be the side of $\mathcal{F}$ different from $\sigma_1^{-1} L_1$ containing the vertex $v_2=\sigma_1^{-1} v$ and let $\sigma_2\in \mathcal{S}(\mathcal{F})$ be the associated side pairing transformation. Continuing like this yields a periodic sequence of pairs 
$$(L_1,v_1), (L_2,v_2),\ldots,$$
with minimal period $m\geq 1$, say, which we call a \emph{cycle of $\mathcal{F}$}. It is now clear that $\sigma_1\sigma_2\ldots \sigma_m$ fixes $v_1$ and thus if we put $\nu=|\Gamma_{v_1}|$ (i.e. the size of the stabilizer of $v_1$ inside $\Gamma$) then we get the following relation on the side pairing transformations;
$$ (\sigma_1\sigma_2\ldots \sigma_m)^\nu=1, $$
where if $\nu=\infty$ (i.e. $v_1$ is a boundary vertex) this is understood as the empty relation. Notice that the relation does not depend on the choice of starting point $(L_1,v_1)$. We call these the \emph{cycle relations of $\mathcal{F}$} and we have the following key theorem of Poincar\'{e}, see \cite{Maskit71} for a proof.
\begin{lemma}[Poincar\'{e}'s Theorem] \label{lem:poincare}The side pairing transformations $\mathcal{S}(\mathcal{F})$ generate $\Gamma$, and the inverse and cycle relations give a complete set of relations for $\mathcal{S}(\mathcal{F})$.
\end{lemma}
Secondly it is clear that each side appears in exactly one cycle and thus we obtain the following useful fact.
\begin{lemma}[cf. Lemma 5.3 of \cite{Voight09}]\label{lem:voight}
Each side pairing transformation appears exactly once in a cycle relation.
\end{lemma}

\subsubsection{Bounding coordinates of side pairing transformations} 
Now we consider the image of the inverse and cycle relations when mapped to $V_p$ under the composition 
$$\Gamma_0(p)\twoheadrightarrow \Gamma_0(p)^\mathrm{ab}\twoheadrightarrow H_1(Y_0(p),\Z)\hookrightarrow V_p=H_1(Y_0(p),\R)\cong \R^{2g+1}.$$ 
Notice that since $V_p$ is torsion-free we can divide the relations by $\nu$. Combining the inverse and cycle relations with the choice of basis $B=\{v_0,\ldots, v_{2g}\}\subset V_p$ (corresponding to side pairing transformations under (\ref{eq:projab})) gives rise to a system of linear equations in $V_p$;
$$\mathcal{L}_{\mathcal{F},B}(x_0,\ldots, x_n)= \left\{ \sum_{0\leq i\leq n} a_{ik}x_i+\sum_{0\leq j\leq 2g} b_{jk}v_j=0 \right\}_{0\leq k \leq K},$$
where $x_0,\ldots, x_n$ are variables, one for each pair $\{\sigma,\sigma^{-1}\}\subset \mathcal{S}(\mathcal{F})$ not corresponding to an element of $B$. Notice that Lemmata \ref{lem:poincare} and \ref{lem:voight} translate to the two key properties:
\begin{enumerate}
\item\label{item:1v} $\mathcal{L}_{\mathcal{F},B}$ has exactly one solution $(x_0,\ldots, x_n)\in (V_p)^{n+1}$.
\item\label{item:2v} For any subset $A\subset \{0,\ldots, K\}$, $0\leq i\leq n$ and $0\leq j\leq 2g$, we have
$$ \sum_{k\in A}a_{ik},\, \sum_{k\in A}b_{jk} \in \{-1,0,1\}. $$
\end{enumerate}
We have the following general result about such systems of linear equations.
\begin{lemma}\label{lem:linear}
Let  
$$\mathcal{L}(x_0,\ldots, x_n)= \left\{ \sum_{0\leq i\leq n} a_{ik}x_i+\sum_{0\leq j\leq 2g} b_{jk}v_j=0 \right\}_{0\leq k \leq K},$$ 
 be a system of linear equations in $V_p$ satisfying (\ref{item:1v}) and (\ref{item:2v}).
 
Then the unique solution $(x_0,\ldots, x_n)\in (V_p)^{n+1}$ satisfies
\begin{equation}\label{eq:induction}x_i=\sum_{0\leq j\leq 2g} c_{ij} v_j,\quad c_{ij}\in \{-1,0,1\},\end{equation}
for all $0\leq i\leq n$.
\end{lemma}
\begin{proof}
We proceed by induction on $K+1\geq 0$ (i.e. the number of equations). If $K+1=0$ then there is nothing to prove. Now assume the claim is known for systems of $K'<K+1$ equations. We start by making the following reductions; we may assume that all equations contain a variable (i.e. $\forall k \exists i:a_{ik}\neq 0$), since otherwise we can remove such an equation. Furthermore, we may assume that there some variable appearing exactly once (i.e. $\exists i : |\{k: a_{ik}\neq 0\}|=1$). 
To see this observe that if this is not the case then by (\ref{item:1v}) and (\ref{item:2v}) every variable appears exactly twice with coefficients $\pm 1$, respectively. This means that the equation obtained by adding all of the equation in $\mathcal{L}(x_0,\ldots, x_n)$ has to be trivial. Thus by removing any equation, we obtain an equivalent system with $K$ equation and thus we are done by the induction hypothesis. 

The key observation is  that there is always an equation with exactly one variable appearing (i.e. $\exists k : |\{i: a_{ik}\neq 0\}|=1$). If not, then since every variable appears at most twice and there is one variable appearing once (by the above reductions) we have a system of linear equation with more variables than equations. Thus the number of solutions $(x_0,\ldots, x_n)\in (V_p)^{n+1}$ is either $0$ or $\infty$, which contradicts (\ref{item:1v}). 
Let $x_i$ be the only variable appearing in some equation. Then this equation determines $x_i$ which by (\ref{item:2v}) satisfies (\ref{eq:induction}). Now if $-x_i$ does not appear we are done by the induction hypothesis. Otherwise we add the equation containing $x_i$ to the one containing $-x_i$ which gives a system of $K$ linear equations again satisfying (\ref{item:1v}) and (\ref{item:2v}). Thus the claim follows from the induction hypothesis.
\end{proof}
\begin{cor}\label{lem:relations}
Let $\Gamma$ be a Fuchsian group with a fundamental polygon $\mathcal{F}$, and let $B$ be a basis for $H_1(\Gamma\backslash \Hb, \R)$ consisting of classes of the form $\{z,\sigma z\}$ with $\sigma\in \mathcal{S}(\mathcal{F})$ and $z\in \Hb$. 

Then we have for any $\sigma \in \mathcal{S}(\mathcal{F})$, $z\in \Hb$ and $\omega\in  B^{\ast}$ that
$$ |\langle \{z, \sigma z\}, \omega \rangle|\leq 1.$$
\end{cor}
\begin{proof} 
Observe that if the variable $x_i$ corresponds to $\{\sigma,\sigma^{-1}\}\in \mathcal{S}(\mathcal{F})$ then the numbers
$$(\langle \{ z, \sigma z\}, \omega \rangle)_{\omega\in B^{\ast}},$$ are exactly the coordinates of either $x_i$ or $-x_i$ in the basis $B$ where $(x_0,\ldots, x_n)\in (V_p)^{n+1}$ is the unique solution to $\mathcal{L}_{\mathcal{F},B}$. Now the result follows directly from Lemma \ref{lem:linear}.
\end{proof}


\subsection{Zagier's fundamental polygon}\label{sec:zagier}
We will now consider a fundamental polygon for $\Gamma_0(p)$ introduces by Zagier \cite[Section 3]{Zagier85} which will give rise a set of natural bases for the homology. The following is a  fundamental polygon for $\Gamma_0(p)$; 
$$\tilde{\mathcal{F}}_\mathrm{Zag}(p)=\cup_{i=0}^{p} \sigma_i \mathcal{F}_\mathrm{std},$$ 
where $\sigma_i=\begin{psmallmatrix} 0 & -1 \\ 1 & i \end{psmallmatrix}$ for $0\leq i <p$ and $\sigma_p=\begin{psmallmatrix} 1 & 0 \\ 0 & 1 \end{psmallmatrix}$ and 
\begin{equation}\label{eq:std}\mathcal{F}_\mathrm{std}:=\{z\in \Hb: |\Re z|<1/2, |z|>1\},\end{equation}
 is the standard fundamental polygon for $\PSL_2(\Z)$. This gives rise to the side pairing transformation set
$$\mathcal{S}(\tilde{\mathcal{F}}_\mathrm{Zag}(p))=  
\{\begin{psmallmatrix} 1 & 0 \\ \pm p & 1 \end{psmallmatrix}\}\cup\{\begin{psmallmatrix} 1 & \pm 1 \\ 0 & 1 \end{psmallmatrix}\}\cup
\left\{\begin{psmallmatrix} -a^\ast & -1 \\ (aa^\ast+1) & j \end{psmallmatrix} : 0<a<p\right\}, $$
where $0<a^*<p$ is such that $aa^*\equiv -1 \modulo p$. Recall that $\Gamma_0(p)$ is normalized by the matrix $W_{p}=\begin{psmallmatrix} 0 & -1/\sqrt{p} \\ \sqrt{p} & 0 \end{psmallmatrix}$, which implies that also $\mathcal{F}_\mathrm{Zag}(p):=W_{p} \tilde{\mathcal{F}}_\mathrm{Zag}(p)$ is a fundamental polygon for $\Gamma_0(p)$ with side pairing transformations given by
\begin{align*}\mathcal{S}(\mathcal{F}_\mathrm{Zag}(p))&=  
W_{p}\left(\{\begin{psmallmatrix} 1 & 0 \\ \pm p & 1 \end{psmallmatrix}\}\cup \{\begin{psmallmatrix} 1 & \pm 1 \\ 0 & 1 \end{psmallmatrix}\}\cup
\left\{\begin{psmallmatrix} -a^\ast & -1 \\ (aa^\ast+1) & a \end{psmallmatrix} : 0<a<p\right\}\right)W_{p}^{-1}\\
&= \{\begin{psmallmatrix} 1 & \pm 1 \\ 0 & 1 \end{psmallmatrix}\}\cup \{\begin{psmallmatrix} 1 & 0 \\ \pm p & 1 \end{psmallmatrix}\}\cup
\left\{\begin{psmallmatrix} a & -(aa^\ast+1)/p \\ p & -a^\ast \end{psmallmatrix} : 0<a<p\right\}, \end{align*}
which by Lemma \ref{lem:poincare} generate $\Gamma_0(p)$ (as claimed in the introduction). Both of these fundamental polygons have the nice property that all cuspidal sides (i.e. sides containing a cusp) are paired by a parabolic element of $\Gamma_0(p)$ (which is not the case for all fundamental polygons). The elements of $\mathcal{S}(\mathcal{F}_\mathrm{Zag}(p))$ are minimal in the sense that the archimedean sizes of the entries are as small as one can hope for ($\leq p$). 
Explicitly, the fundamental polygon $\mathcal{F}_\mathrm{Zag}(p)$ has $p+3$ vertices; the two cusps $0$ and $\infty$ as well as $\frac{2j-1+i\sqrt{3}}{2p}$ for $0\leq j\leq p$. We see that the matrix $W_{p}$ takes the hyperbolic triangle with vertices $ \{\tfrac{-1+i\sqrt{3}}{2p},0,\tfrac{1+i\sqrt{3}}{2p}\}$ to the hyperbolic triangle with vertices $\{ \tfrac{-1+i\sqrt{3}}{2},\infty,\tfrac{1+i\sqrt{3}}{2}\}$. In particular, the subgroup $\langle W_{p}, \Gamma_0(p)\rangle \leq \PSL_2(\R)$ has a fundamental domain contained in 
\begin{equation} \label{eq:funddomainHp} \{ z\in \Hb : |\Re z|\leq 1/2, \Im z\geq \sqrt{3}/(2p) \},\end{equation}  
 which will be useful later on.

We define the following compatible family of bases of the homology groups.
\begin{defi}\label{defi:bsc}A \emph{basic basis of level $p$} is a  basis  ${B}\subset H_1(Y_0(p),\R)$ consisting of elements $\{z,\sigma z\}$ with $z\in \Hb$ and $\sigma \in \mathcal{S}(\mathcal{F}_\mathrm{Zag}(p))$.
\end{defi}

\subsection{Special fundamental polygons}\label{sec:kul}
By the Kurosh subgroup theorem we know that any subgroup of $\PSL_2(\Z)\cong \Z/2\Z\ast \Z/3\Z$ is isomorphic to a free product of a number of copies of $\Z/2\Z,\Z/3\Z$ and $\Z$. In particular, a torsion-free Hecke congruence subgroup $\Gamma_0(p)$ is a free group on $k=\mathrm{rank}_{\Z}\, \Gamma_0(p)^\mathrm{ab}$ generators. We will now describe an explicit geometric way due to Kulkarni \cite{Kulkarni91} for constructing a set of independent generators of $\Gamma_0(p)$. The starting point  are  so-called \emph{special fundamental polygons} of $\Gamma_0(p)$. 

Let $g$ be the genus of $Y_0(p)$, and $e_2,e_3$ the number of conjugacy classes of subgroups in $\Gamma_0(p)$ of order respectively $2$ and $3$.  Following \cite{Kulkarni91}, we define a \emph{Farey symbol of level $p$} as a sequence of reduced fractions;
$$\frac{0}{1}=\frac{a_{0}}{b_{0}}<\frac{a_1}{b_1}<\ldots<\frac{a_{n-1}}{b_{n-1}}< \frac{a_{n}}{b_{n}}=\frac{1}{1} , $$
with $n=4g+e_2+e_3$ such that $a_{i+1}b_{i}-a_{i}b_{i+1}=1$ for all $1\leq i<n$ and furthermore (considering below the indices modulo $n+2$) 
\begin{itemize}
\item there are $e_2$ even indices $i$ such that 
$$ b_i^2+b_{i+1}^2 \equiv 0 \modulo p, $$
\item there are $e_3$ odd indices $i$ such that 
$$ b_i^2+b_ib_{i+1}+b_{i+1}^2 \equiv 0 \modulo p, $$
\item for the remaining $4g$ free indices there is a pairing $i\leftrightarrow i^\ast$ satisfying 
$$ b_ib_{i^\ast}+b_{i+1}b_{i^\ast+1}\equiv 0 \modulo p.  $$ 
\end{itemize}
Such a symbol always exists and one can even find one which is symmetric around $1/2$ \cite[Section 13]{Kulkarni91}.  Dooms--Jesper--Konolalov \cite{DoomsJespersKonovalov10} have described an algorithm for determining Farey symbols of general level. 

Consider the polygon $\mathcal{P}(p)$ with vertices at $\infty$, at the fractions of the Farey symbol, at the midpoint of the geodesic circle connecting $a_i/b_i$ and $a_{i+1}/b_{i+1}$ for $i$ an even index, and for an odd index $i$ at the $\PGL_2(\Z)$-translate of $\frac{1+i\sqrt{3}}{2}$ lying between $a_i/b_i$ and $a_{i+1}/b_{i+1}$ (for details see \cite[Section 2]{Kulkarni91}). Note that $\mathcal{P}(p)$ consists of $\PGL_2(\Z)$-translates of 
\begin{equation}\label{eq:F+}\mathcal{F}^+:=\{z\in \Hb: 0<\Re z<1/2, |z|>1\}.
\end{equation}
The map $\ast$ defines a side pairing transformation on this polygon by identifying the halfcircle connecting $\frac{a_{i}}{b_{i}},\frac{a_{i+1}}{b_{i+1}}$ and the one connecting $\frac{a_{i^\ast+1}}{b_{i^\ast+1}},\frac{a_{i^\ast}}{b_{i^\ast}}$, as well as identifying the vertical sides of $\mathcal{P}(p)$ and the elliptic sides. By Poincar\'{e}'s theorem, $\mathcal{P}(p)$ together with this pairing defines a subgroup of $\PSL_2(\R)$ which can be shown to be equal to $\Gamma_0(p)$. Furthermore, since $\mathcal{P}(p)$ has a minimal number of sides, it follows that an independent set of generators of $\Gamma_0(p)$ is given by the matrices which maps between the sides identified by the pairing induces from $\ast$. These matrices are explicitly given by  (see \cite[Theorem 6.1]{Kulkarni91})
\begin{align}
\label{eq:basisT}T=\begin{psmallmatrix} 1 & 1 \\ 0 & 1\end{psmallmatrix},\end{align} 
and the $e_2$ matrices of order $2$ and $e_3$ matrices of order $3$
\begin{align}
\label{eq:basise2} &\begin{pmatrix} a_{i+1}b_{i+1}+a_{i}b_i & -a_i^2-a_{i+1}^2 \\ b_i^2+b_{i+1}^2 & -a_{i+1}b_{i+1}-a_{i}b_{i}\end{pmatrix},\\
\label{eq:basise3} &\begin{pmatrix} a_{i+1}b_{i+1}+a_{i}b_{i+1}+a_{i}b_i & -a_i^2-a_ia_{i+1}-a_{i+1}^2 \\ b_i^2+b_ib_{i+1}+b_{i+1}^2 & -a_{i+1}b_{i+1}-a_{i+1}b_i-a_{i}b_{i}\end{pmatrix},
 \end{align} 
together with the $2g$ hyperbolic matrices  
\begin{align}\label{eq:basisK2} \begin{pmatrix} a_{i^\ast+1}b_{i+1}+a_{i^\ast}b_i & -a_ia_{i^\ast}-a_{i+1}a_{i^\ast+1} \\ b_ib_{i^\ast}+b_{i+1}b_{i^\ast+1} & -a_{i+1}b_{i^\ast+1}-a_{i}b_{i^\ast}\end{pmatrix} \text{ for $i<i^\ast$  a pair}. \end{align}  
Observe that the $2g$ hyperbolic matrices above together with $T$ define a basis for $H_1(Y_0(p),\R)$ under the map (\ref{eq:projab}).
Recently, Doan--Kim--Lang--Tan \cite{DKLPT22} have shown that one can find \emph{minimal special fundamental polygons} $\mathcal{P}_\mathrm{min}(p)$ meaning that we have 
$$b_ib_{i^\ast}+b_{i+1}b_{i^\ast+1}=p, \quad 0\leq i\leq n,$$
which implies that in fact $\mathcal{S}(\mathcal{P}_\mathrm{min}(p))\subset \mathcal{S}(\mathcal{F}_\mathrm{Zag}(p))$. 



\section{Proof of Theorem \ref{thm:std3}}
In this section we will prove the following statement which implies our first main result. 

\begin{thm}\label{thm:std}
Fix a prime $p$ and $\delta \in (0,\tfrac{1}{2114})$. Consider a real quadratic field $K$ of discriminant $d_K $ such that $p$ splits in $K$ with $p\mathcal{O}_K=\mathfrak{p}_1\mathfrak{p}_2$. Consider a subgroup $H\leq \Cl_K^+$ with $\mathfrak{p}_1 \notin H$ and $J=(\sqrt{d_K})\notin H$. 

Then we have 
for any $\omega\in H^1(Y_0(p), \R)$ we have as $d_K \rightarrow \infty$
\begin{equation}\label{eq:weakconv} \frac{\sum_{A\in H} \langle [\mathcal{C}_{A}(p)], \omega \rangle}{|\sum_{A\in H} \langle [\mathcal{C}_{A}(p)], \omega_{E}(p) \rangle|}= -\langle v_{E}(p), \omega\rangle+O_\omega(d_K^{-\delta}), \end{equation}
where $\langle \cdot, \cdot\rangle$ denotes the cap product pairing, and $v_{E}(p),\omega_{E}(p)$ are the Eisenstein classes in (co)homology defined in (\ref{eq:eisensteinclass}) and (\ref{eq:eisensteinclasshom}).  
 \end{thm}
First of all let us see how Theorem \ref{thm:std3} follows from this.

\begin{proof}[Proof of Theorem \ref{thm:std3} assuming Theorem \ref{thm:std}]
Let $B$ be any basis of $V_p=H_1(Y_0(p),\R)$ containing $v_E(p)$ and let $ B^{\ast}$ denote the dual basis of $H^1(Y_0(p),\R)$ with respect to the cap product pairing. Consider the isomorphism $V_p\cong \R^{2g+1}$ defined by sending $B$ to the standard basis of $\R^{2g+1}$ and denote by $|\!|\cdot |\!|=|\!|\cdot |\!|_{B,\infty}$ the norm on $V_p$ obtained by pulling back the sup norm with respect to the standard basis of $\R^{2g+1}$. We will prove the convergence (\ref{eq:convergenceformulation}) which implies Theorem \ref{thm:std3}. By Theorem \ref{thm:std}, we have for $d_K$ large enough that
\begin{align}
|\!|\sum_{A\in H}[\mathcal{C}_{A}(p)]|\!|=\max_{\omega \in  B^{\ast}} |\langle \sum_{A\in H}  [\mathcal{C}_{A}(p)], \omega\rangle |&= \left| \sum_{A\in H} \langle [\mathcal{C}_{A}(p)], \omega_{E}(p) \rangle \right| \max_{\omega \in  B^{\ast}} \left(|\langle v_{E}(p), \omega\rangle |+O_\omega(d_K^{-\delta})\right)\\
\label{eq:supnormcalc}&= \left| \sum_{A\in H} \langle [\mathcal{C}_{A}(p)], \omega_{E}(p) \rangle \right| (1+O_B(d_K^{-\delta})).
\end{align}
By the triangle inequality, we conclude that
\begin{align}\label{eq:translationtonorm}
\left|\!\left|\frac{\sum_{A\in H}[\mathcal{C}_{A}(p)]}{|\!|\sum_{A\in H}[\mathcal{C}_{A}(p)]|\!|}+v_{E}(p) \right|\!\right|
\leq \left|\!\left|\frac{\sum_{A\in H}[\mathcal{C}_{A}(p)]}{|\langle \sum_{A\in H}  [\mathcal{C}_{A}(p)], \omega_E(p)\rangle|}+v_{E}(p) \right|\!\right|+ \left| \frac{|\!|\sum_{A\in H}[\mathcal{C}_{A}(p)]|\!|}{|\langle \sum_{A\in H}  [\mathcal{C}_{A}(p)], \omega_E(p)\rangle|}-1\right|.
\end{align}
Finally by (\ref{eq:supnormcalc}) and Theorem \ref{thm:std}, the above is bounded by $O_B(d_K^{-\delta})$ which yields the wanted expression.   
\end{proof}


The rest of this section is occupied with the proof of Theorem \ref{thm:std}. The idea is to do a change of coordinates to the Hecke basis (\ref{eq:Heckebasis}) of the real cohomology. Then using formulas due to Hecke and Waldspurger (more precisely an explicit extension due in this case to Popa \cite{Popa08}), we reduce the problem to a question about certain special values of $L$-function. Now the result follows upon applying subconvexity bound (as well as lower bounds for $L$-functions on the critical line). This is exactly the same proof structure as is used in the automorphic approach to Duke's Theorem (see e.g. \cite{MichelVenk06} and the references therein). 

Given $\omega \in H^1(Y_0(p), \R)$, we obtain by expanding in the Hecke basis for homology defined in (\ref{eq:Heckebasisdual}) that 
\begin{align*}
\langle [\mathcal{C}_{A}(p)], \omega \rangle&=\langle v_{E}(p), \omega \rangle \langle [\mathcal{C}_{A}(p)], \omega_{E}(p) \rangle + \sum_{f\in \mathcal{B}_p, \pm} \langle v_f^{\pm}, \omega\rangle \langle  [\mathcal{C}_{A}(p)], \omega_f^{\pm} \rangle.
\end{align*}  	
   
Now we want to average over cosets $CH \subset \Cl_K^+$ for subgroups $H\leq \Cl_K^+$ of the narrow class group. Using a standard trick from Fourier analysis, we can write:
\begin{align}\nonumber  \sum_{A\in CH} \langle [\mathcal{C}_{A}(p)], \omega \rangle &= \frac{|H|}{|\Cl_K^+|}\sum_{\chi \in \widehat{\Cl_K^+}: \chi_{|H}=1 } \chi(C)\sum_{A\in \Cl_K^+} \langle [\mathcal{C}_{A}(p)], \omega \rangle  \overline{\chi}(A) \\
\nonumber &= \langle v_{E}(p), \omega\rangle \sum_{A\in CH} \langle [\mathcal{C}_{A}(p)], \omega_{E}(p) \rangle  \\
\label{eq:Wald1}&+\sum_{f\in \mathcal{B}_p, \pm} \langle v_f^{\pm}, \omega\rangle \frac{|H|}{|\Cl_K^+|}\sum_{\chi \in \widehat{\Cl_K^+}: \chi_{|H}=1 } \chi(C)\sum_{A\in \Cl_K^+} \langle [\mathcal{C}_{A}(p)], \omega_f^{\pm} \rangle \overline{\chi}(A)  \end{align}
The inner sums in the cuspidal contribution are twisted Weyl sums, which we will have to estimate. These sums fits into the frame work of \emph{toric periods} (see e.g. \cite[Section 4]{EiLiMiVe11}) and they turn out to be related to special values of automorphic $L$-functions as firstly proved by Waldspurger \cite{Waldspurger85}. This is exactly the reason why we have changed to the Hecke coordinates.
\subsection{The Eisenstein contribution}
In order to estimate the Eisenstein contribution we will rely on the following classical formula of Hecke (see \cite[(68)]{DuImTo18});
$$\sum_{A\in \Cl_K^+} \int_{\mathcal{C}_{A}(p)} E^*_2(z)dz\,  \chi(A)= 6(1-\chi(J))L(\chi,0) ,$$
where $E^*_2(z)=E_2(z)-\frac{3}{\pi \Im z}$ (with $E_2(z)$ defined in (\ref{eq:E2})), $J=(\sqrt{D})\in \Cl_K^+$ is the different of $K$, and $L(\chi,s)$ denotes the (finite part) of the Hecke $L$-function associated to the narrow class group character $\chi$. Thus by a change of variable we get  
$$\sum_{A\in \Cl_K^+} \int_{\mathcal{C}_{A}(p)} pE^*_2(pz)dz \, \chi(A)=\sum_{A\in \Cl_K^+} \int_{p\,\mathcal{C}_{A}(p)} E^*_2(z)dz \, \chi(A) .$$
We notice by direct computation that $p\,\mathcal{C}_{A}(p)$ (i.e. the dilation of $\mathcal{C}_{A}(p)$ by the factor $p$) is again a closed geodesic as follows; if $\mathcal{C}_{A}(p)$ corresponds to the quadratic form $ax^2+bxy+cy^2$ of discriminant $d_K $ and level $p$, then $p\,\mathcal{C}_{A}(p)$ corresponds to $\tfrac{a}{p}x^2+bxy+cpy^2$ (which now might not be of level $p$). Recall that in Section \ref{sec:geodesic} we fixed a residue $r\modulo 2p$ such that $r^2\equiv d_K\modulo 4p$. One can now check by direct computation on ideals that 
$$\left(\Z \frac{a}{p} + \Z \frac{b - \sqrt{d_K}}{2}\right) \left(\Z p + \Z \frac{r - \sqrt{d_K}}{2}\right) = \Z a + \Z \frac{b - \sqrt{d_K}}{2}$$
if $a > 0$, and similarly
$$\left(\Z \left(-\frac{a}{p}\sqrt{d_K}\right) + \Z \frac{d - b\sqrt{d_K}}{2}\right) \left(\Z  p + \Z \frac{r - \sqrt{d_K}}{2}\right) = \Z (-a\sqrt{d_K}) + \Z \frac{d - b\sqrt{d_K}}{2}$$
if $a < 0$. This implies that when projected to the modular curve $\PSL_2(\Z)\backslash \Hb$ of level $1$, we have  
$$ p\, \mathcal{C}_{A}(p)=\mathcal{C}_{AA_p}(1),$$
where $A_p=[p, \frac{r-\sqrt{d_K}}{2}]\in \Cl_K^+$ (using the notation (\ref{eq:alphabeta})). This implies that
\begin{align}\nonumber \sum_{A\in \Cl_K^+} \langle [\mathcal{C}_{A}(p)],\omega_{E}(p)\rangle \chi(A)&=\sum_{A\in \Cl_K^+} \int_{\mathcal{C}_{A}(p)} E_{2,p}(z)dz\,\, \chi(A)\\
\label{eq:Eisformp}&= \frac{6}{p-1}(1-\chi(J))(\overline{\chi(A_p)}-1)L(\chi,0), \end{align}
using that $(p-1)E_{2,p}(z)=pE_{2}^*(pz)-E_{2}^*(z)$. Note that the above vanishes when $\chi(J)=1$ or $\chi(A_p)=1$. 

By the functional equation for class group $L$-functions \cite[p. 13]{DuImTo18} and the fact class group characters are self-dual (i.e. $\chi\circ c=\overline{\chi}$ where $c$ denotes complex conjugation) we conclude 
\begin{equation}\label{eq:selfdual} \overline{L(\chi,0)}=L(\overline{\chi},0)=L(\chi\circ c,0)=L(\chi,0)=\pi^{-2} d_K^{1/2} L(\chi,1)>0.,\end{equation}
By standard estimates for $L$-functions on the critical line we have that for all $\eps>0$ there is some uniform (but inefficient) constant $c_\eps>0$ such that
\begin{equation}\label{eq:Siegel} L(\chi,0)\geq c_\eps d_K^{1/2-\eps}.\end{equation}


\begin{prop}\label{prop:lowerbound} Let $H\leq \Cl_K^+$ be a subgroup such that $A_p\notin H$ and $J\notin H$. Then for each $\eps>0$ there exists a constant $c_\eps>0$ such that  
\begin{equation}\label{eq:lowerbound}  \sum_{A\in H} \langle [\mathcal{C}_{A}(p)], \omega_{E}(p) \rangle \leq -c_\eps \frac{d_K^{1/2-\eps}}{p}.  \end{equation}
\end{prop}
\begin{proof}
By orthogonality, the Hecke formula, and equations (\ref{eq:selfdual}) and (\ref{eq:Siegel}), we conclude 
 \begin{align} \label{eq:genuslowerbnd} \sum_{A\in H} \langle [\mathcal{C}_{A}(p)], \omega_{E}(p) \rangle&=\frac{|H|}{|\Cl_K^+|}\sum_{\chi\in \widehat{\Cl_K^+}: \chi_{|H}=1}\sum_{A\in \Cl_K^+} \langle [\mathcal{C}_{A}(p)], \omega_{E}(p) \rangle\chi(A)\\
 &=\frac{12|H|}{(p-1)|\Cl_K^+|}\sum_{\substack{\chi\in \widehat{\Cl_K^+}:\\ \chi_{|H}=1,\chi(J)=-1}}\sum_{A\in \Cl_K^+} (1-\Re \chi(A_p))L(\chi,0)\\
 &\leq -c_\eps \frac{d_K^{1/2-\eps}}{p}\frac{|H|}{|\Cl_K^+|}\sum_{\substack{\chi\in \widehat{\Cl_K^+}:\\ \chi_{|H}=1,\chi(J)=-1}} (1-\Re \chi(A_p)),  \end{align}
 where we are using that $(1-\Re \chi(A_p))\geq 0$. Now we observe that if $H':=\langle H,J\rangle\leq \Cl_K^+$ denotes the group generated by $H\leq \Cl_K^+$ and $J=(\sqrt{d_K})$ then we can write 
\begin{align} \sum_{\substack{\chi\in \widehat{\Cl_K^+}:\\ \chi_{|H}=1,\chi(J)=-1}} (1-\Re \chi(A_p))
&= \frac{|\Cl_K^+|}{|H'|}
-\Re \left( \sum_{\substack{\chi\in \widehat{\Cl_K^+/H}}} \chi(A_p H)-\sum_{\substack{\chi\in \widehat{\Cl_K^+/H'}}}\chi(A_p H')  \right)\\
&= \begin{cases} \tfrac{|\Cl_K^+|}{2|H|},& A_p\in H'\\
      \frac{|\Cl_K^+|}{|H|},& A_p\notin H'\end{cases}.\end{align}
Here we are using that $J\notin H$, which implies that $|H'|=2|H|$, and  $A_p\notin H$. Inserting this into the above yields the wanted estimate.
 \end{proof}
\subsection{The cuspidal contribution}
In the case of cuspidal Weyl sums, we will use the following formula due to Popa \cite[Theorem 6.3.1]{Popa06} (see also the first formula in the introduction of \cite{Popa06});
\begin{align*}
\left|\sum_{A\in \Cl_K^+} \langle [\mathcal{C}_{A}(p)], \omega_f \rangle \overline{\chi}(A)\right|^2=(4\pi^2)^{-1}d_K^{1/2}L(\pi_f\otimes \pi_\chi,1/2), 
\end{align*}
where $\pi_f$ denotes the $\GL_2$-automorphic representation associated to $f$, $\pi_\chi$ denotes the $\GL_2$-automorphic representation associated to $\chi$ via automorphic induction and $L(\pi_f\otimes \pi_\chi,1/2)$ denotes the central value of the (finite part of the) Rankin--Selberg $L$-function of $\pi_f$ and $\pi_\chi$. This implies
\begin{align}\label{eq:Wald2}
\sum_{A\in \Cl_K^+} \langle [\mathcal{C}_{A}(p)], \omega_f^{\pm} \rangle \overline{\chi}(A)
= \frac{d_K^{1/4}}{2}\left( \epsilon_{f,\chi} |L(\pi_f\otimes \pi_\chi,1/2)|^{1/2}\pm \epsilon_{f,\overline{\chi}} |L(\pi_f\otimes \pi_{\overline{\chi}},1/2)|^{1/2}  \right),
\end{align}    
with $|\epsilon_{f,\chi}|=|\epsilon_{f,\overline{\chi}}|=1$. By the subconvexity bound 
$$L(\pi_f\otimes \pi_{\chi},1/2)\ll_f d_K^{1/2-1/1057},$$ 
due to Michel \cite[Theorem 2]{Michel04}, we conclude the following bound on the twisted Weyl sums; 
\begin{equation}\label{eq:Weylbnd}\sum_{A\in \Cl_K^+} \langle [\mathcal{C}_{A}(p)], \omega_f^{\pm} \rangle \overline{\chi}(A)\ll_f d_K^{1/2-1/2114}.\end{equation}

Combining this with the lower bound for the Eisenstein contribution yields the proof of our main result.
\begin{proof}[Proof of Theorem \ref{thm:std}] Let $H\leq \Cl_K^+$ be a subgroup as in the statement. Then by Proposition \ref{prop:lowerbound} we have for all $\eps>0$ that there exists an absolute constant $c_\eps>0$ such that  
\begin{equation}\label{eq:dividingthr}\left|\sum_{A\in H} \langle [\mathcal{C}_{A}(p)], \omega_{E}(p) \rangle\right|=-\sum_{A\in H} \langle [\mathcal{C}_{A}(p)], \omega_{E}(p) \rangle\geq c_\eps  d_K^{1/2-\eps}p^{-1}.\end{equation}
Now the result follows from (\ref{eq:Wald1}) (with $C=1$) by dividing through by (\ref{eq:dividingthr}) and bounding the cuspidal contribution using (\ref{eq:Weylbnd}). Here we are using that the number of narrow class group characters $\chi$ such that $\chi_{|H}=1$ is exactly $\tfrac{|\Cl_K^+|}{|H|}$, and the remaining terms in (\ref{eq:Wald1}) do not depend on $d_K$.\end{proof}

\section{A homological version of the sup norm problem}\label{sec:amplpretrace}
In this section we will deal with the problem of obtaining a version of Theorem \ref{thm:std} where the level $p$ is allowed to vary. By (\ref{eq:Wald1}) and (\ref{eq:Wald2}) this requires first of all bounds for $L(\pi_f\otimes \pi_{\chi},1/2)$ in terms of both $d_K $ and $p$ which in the case of genus characters has been studied by Petrow and Young \cite{PeYo19}. Secondly we will need estimates for the cap product pairings $\langle v_f^{\pm}, \omega\rangle$ (with $f\in \mathcal{B}_p$) in terms of $p$. This can be thought of as an analogue of the \emph{sup norm problem} from \emph{arithmetic quantum chaos} (see e.g. \cite{BlomerHolo10}) as we will explain below. We will have to restrict to certain compatible families of $\omega\in H^1(Y_0(p),\R)$,  which in our case we will be dual bases of basic bases of level $p$ as in Definition \ref{defi:bsc}.

\begin{thm}\label{thm:maincontr}
Let $p$ be prime and let ${B}\subset H_1(Y_0(p),\Z)$ be a basic basis of level $p$. Then for $\omega\in  B^{\ast}$ (the dual basis of $B$ with respect to (\ref{eq:cap})) we have
\begin{equation}\label{eq:2ndmomentboundMin}\sum_{\epsilon\in \{\pm\}}\sum_{f\in \mathcal{B}_p} | \langle v_f^{\epsilon}, \omega\rangle|^2 \ll_\eps p^{1+\eps} .\end{equation}
\end{thm}   

See Remark \ref{rem:conj} below for thoughts on the optimal bound that one can expect. Notice that since the newforms $f\in \mathcal{B}_p$ are Hecke normalized, the dual vectors $v_f^{\pm}$ are very subtle quantities as they are related to the \emph{minimal periods} $c_f^{\pm}$ of $f$ (characterized when $f$ has rational coefficients by $(c_f^{+})^{-1} \Re f(z)dz\in H^1(X_0(p), \Z)$ being primitive and similarly for $c_f^{-}$). Upper bounding $c_f^{\pm}$ in terms of the level is extremely hard as any polynomial bound implies (a weak form of) the ABC-conjecture (see \cite{Goldfeld02}). 

To put Theorem \ref{thm:maincontr} into perspective it is again useful to compare to the imaginary quadratic analogue of supersingular reduction of CM elliptiuc curves. Recall that in Secftion \ref{sec:CM} we defined an $n$-dimensional vector space $H_0(X^{p,\infty},\R)$ having a canonical basis $e_1,\ldots, e_n$ corresponding to the connected components of the conic curve $X^{p,\infty}$ as defined in (\ref{eq:conic}) which in turn can be identified with isomorphism classes of supersingular elliptic curves defined over $\mathbb{F}_{p^2}$. There is a canonical bilinear form $\langle \cdot, \cdot \rangle_\mathrm{ss}$ on $H_0(X^{p,\infty},\R)$ given by 
\begin{align}\label{eq:ss}\langle e_i, e_j \rangle_\mathrm{ss}:= \delta_{i,j}w_i,\end{align} 
where $w_i$ is the size of the endomorphism ring of the elliptic curves corresponding to $e_i$. The homology group $H_0(X^{p,\infty},\R)$ carries a natural action of the Hecke algebra (of level $p$) defined via correspondences and these linear operators are self-adjoint with respect to the bilinear form $\langle \cdot, \cdot \rangle_\mathrm{ss}$. The homology group  $H_0(X^{p,\infty},\R)$ is (by the Jacquet--Langlands correspondence) isomorphic as a Hecke algebra to the space $\mathcal{M}_2(p)$ of holomorphic modular forms of level $p$ and weight 2. In particular, one can associate to each Hecke eigenform $f$ a unique element $e_f\in H_0(X^{p,\infty},\R)$ such that $\langle e_f, e_f \rangle_\mathrm{ss}=1$ and $T_\ell\,  e_f=\lambda_f(\ell)e_f$ where $\lambda_f(\ell)$ denotes the $\ell$-th Hecke eigenvalues of $f$ with $\ell\neq p$ prime. In this setting, we are interested in upper bounds for the {\lq\lq}$L^r$-norms{\rq\rq};
$$ |\!|e_f|\!|_r:=\left(\sum_{i=1}^n |\langle e_f, e_i \rangle_\mathrm{ss}|^r\right)^{1/r}, 1\leq r< \infty,\quad |\!|e_f|\!|_\infty:=\max_{1\leq i\leq n}|\langle e_f, e_i \rangle_\mathrm{ss}|.  $$
Putting $r=2$ and using Parseval, one obtains $|\!|e_f|\!|_2=1$ which implies the trivial (or convexity) bound $|\!|e_f|\!|_r\leq 1 $ for all $r\geq 2$. In the case $r=\infty$ Blomer and Michel \cite{BlomerMichel11} were the first to go beyond this by proving $|\!|e_f|\!|_\infty\ll p^{-1/6+\eps}$. This was recently improved by Khayutin--Nelson--Steiner \cite[Corollary 2.3]{KhayutinNelsonSteiner22} who obtained $|\!|e_f|\!|_\infty\ll p^{-1/4+\eps}$. Liu--Masri--Young \cite[Proposition 1.12]{LMY15} obtained the $L^4$-bound $|\!|e_f|\!|_4\ll p^{-1/8+\eps}$. 
 
The fact that the Hecke operators are self-adjoint with respect to the bilinear form (\ref{eq:ss}) on $H_0(X^{p,\infty},\R)$ implies that the dual basis 
$$e_1^\ast,\ldots, e_n^\ast\in H_0(X^{p,\infty},\R)^\ast,$$ 
of $e_1,\ldots, e_n$ (which we can identify with $w_1^{-1}\langle \cdot, e_1\rangle_\mathrm{ss},\ldots , w_n^{-1}\langle \cdot, e_n\rangle_\mathrm{ss}$) have the same behavior under the action of the Hecke algebra. This is why in the proof of the level aspect version (\ref{eq:weakmichel2}) by Liu--Masri--Young one does not need any non-trivial input to bound the factors $\langle e_f, e_i\rangle_\mathrm{ss}$ appearing when spectrally expanding; one can simply employ the trivial Parseval bound $\sum_{f\in \mathcal{B}_p} |\langle e_f, e_i\rangle_\mathrm{ss}|^2\leq 1$.This is not necessarily the case for the basic bases ${B}\subset H_1(Y_0(p), \R)$ in our setting. This means in particular that we do not know whether
\begin{equation}\label{eq:nonpos}\langle v , \omega_f^{\pm}\rangle \langle v_f^{\pm}, v^\ast  \rangle ,\end{equation}
 is positive or not for $v\in {B}$ and $v^\ast\in  B^{\ast}$ such that $\langle v, v^\ast \rangle=1$. This lack of positivity makes it hard to obtain any bound at all.               
 
\subsection{The method of proof}
The first natural approach to bounding the left hand side of (\ref{eq:2ndmomentboundMin}) would be to use a version of the amplified pre-trace formula approach to the sup norm problem in arithmetic quantum chaos (see e.g. \cite{BlomerHolo10}). This is however not very effective due to the possible non-positivity of (\ref{eq:nonpos}) (i.e. we are working in a Banach space rather than a Hilbert space). On the other hand, it has become apparent from the work of Steiner \cite{Steiner20}, Khayutin--Steiner \cite{KhayutinSteiner20}, and Khayutin--Nelson--Steiner \cite{KhayutinNelsonSteiner22} that one in many cases can obtain very strong sup norm bounds by using the theta correspondence. We will employ a version of this argument in our setting. The basis fact is that for $v_0\in H_1(X_0(p),\R)$ and $\omega\in H^1(X_0(p),\R)$	
\begin{equation}\label{eq:gtheta}\sum_{n\geq 1} \langle T_n v_0, \omega\rangle e^{2\pi i nz},\end{equation}
defines a cusp form of level $p$. This follows easily by the fact that the $\pm 1$ isotypic components of $\iota$ acting on $H_1(X_0(p), \R)$ are isomorphic as Hecke modules to the space of holomorphic cusp forms $\mathcal{S}_2(p)$ of weight 2 and level $p$ (which is an incarnation of the Eichler--Shimura isomorphism \cite[Section 8.2]{Sh94}). This can be seen as an instance of the theta correspondence (see \cite[Section 5.3]{DaHaRoVe21}). We will give a simple proof below of the exact statement that we need. Now the key fact is that if $\iota v_0=\pm v_0$ then the $L^2$-norm of (\ref{eq:gtheta}) is (by spectrally expanding) equal to 
$$\sum_{f\in \mathcal{B}_p} |\langle v_0 , \omega_f^{\pm}\rangle|^
2 |\langle v_f^{\pm}, \omega \rangle|
^2.$$   
Note that now we have obtained positivity for free! Thus we are reduced to, one the one hand, a lower bound for $|\langle v_0 , \omega_f^{\pm}\rangle|^
2$ which we resolve in Corollary \ref{cor:lowerbndadd}. And on the other, a bound for the $L^2$-norm of (\ref{eq:gtheta}) which is essentially equal to
\begin{align}\label{eq:L2g}\sum_{1\leq n\leq p} \frac{|\langle T_n v_0, \mathbb{P}_\mathrm{cusp}\,\omega\rangle|^2}{n}.\end{align}
The above is an analogue of the second moment matrix count that appears in e.g. \cite[Section 9]{KhayutinSteiner20}. We will bound (\ref{eq:L2g}) using geometric coding of geodesics. 

\begin{remark}\label{rem:conj}
Let $B$ be a basic basis of level $p$ and let $v\in B$ and $v^\ast\in  B^{\ast}$ such that $\langle v, v^\ast\rangle=1$. Then we get by expanding in the Hecke basis
$$ 1= \sum_{f\in \mathcal{B}_p,\pm } \langle v_f^\pm, v^\ast \rangle \langle v, \omega_f^\pm \rangle.$$
As we will see in Corollary \ref{cor:lowerbndadd} we have on average that $\langle v, \omega_f^\pm \rangle\asymp 1$. Thus if the classes $v_f^\pm$ are perfectly distributed (relative to $B$) then we would have $\langle v_f^\pm, v^\ast \rangle \asymp 1/p$. Thus the strongest sup norm conjecture one can hope for is
$$ \langle v_f^\pm, \omega \rangle \ll_\eps p^{-1+\eps}, $$
for $\omega\in  B^{\ast}$. This might however be too much to hope for.
\end{remark}
\begin{remark}
The dual sup norm problem corresponds to obtaining bounds for the modular symbols
$$\sup_{v\in B}\, |\langle v, \omega_f \rangle|,$$
for $v\in B$ as $p\rightarrow \infty$ where $\omega_f=f(z)dz$ with $f\in \mathcal{B}_p$ a holomorphic Hecke eigenform of weight $2$ and level $p$. Let $v=\{\gamma z,z\}\in H_1(Y_0(p),\Z)$ with $\gamma\in \mathcal{S}(\mathcal{F}_\mathrm{Zag}(p))$ which we may assume is hyperbolic. Picking $z$ on the fixed circle of $\gamma$ and using H\"{o}lder's inequality we see that 
\begin{align*}\left|\langle v, \omega_f \rangle\right| \leq \int_{\mathcal{C}_\gamma} |\Im (z)f(z)|\frac{dz}{\Im (z)}\ll_\eps p^{1/4+\eps} ,
\end{align*}
using the sup norm estimate coming from \cite[Theorem 1.6]{KhayutinNelsonSteiner22}. Here $\mathcal{C}_\gamma$ denotes the closed geodesic on $X_0(p)$ associated to $\gamma$ with hyperbolic length $\ell(\mathcal{C}_\gamma)\ll \log p$. Notice that we also have the trivial bound 
$$ \left( \sum_{f\in \mathcal{B}_p}\left|\langle v, \omega_f \rangle\right|^2\right)^{1/2} \leq \int_{\mathcal{C}_\gamma} \left(\sum_{f\in \mathcal{B}_p} |\Im (z)f(z)|^2\right)^{1/2}\frac{dz}{\Im (z)}\ll_\eps p^{1/2+\eps},  $$ 
by Minkowski's integral inequality and the pre-trace formula (thanks to the referee for pointing this out). 
\end{remark}

\subsection{A lower bound for modular symbols}
In this section we will construct the homology class $v_0\in H_1(X_0(p), \R)$ mentioned above such that $\langle v_0, \omega_f^{\pm}\rangle$ is {\lq\lq}not too small{\rq\rq} in terms of the level. We will be constructing $v_0$ in terms of classes of the shape $\{\tfrac{x}{p},\infty\}$ for $1\leq x< p$ (using the notation (\ref{eq:def})). This reduces the problem to a lower bound for the additive twist $L$-functions of $f$. One could naively try to calculate the average of all additive twists with $1\leq x<p$. But this yields very pour results as there is a lot of cancellation in this sum. Instead we use that the second moment coincides with the second moment of  Dirichlet twists of the $L$-function of $f$ (by the Birch--Stevens formula). By Cauchy--Schwarz it suffices to calculate the first moment of these $L$-functions instead. Such a calculation is quite standard using an approximate functional equation. The exact case we need does however not seem to have been considered before as we are in the case of joint ramification. For the sake of completeness we have provided a detailed argument below. 

To be more precise, let $f\in \mathcal{B}_p$ be a Hecke normalized eigenform and let $\chi$ be a primitive Dirichlet character modulo $q$ with $p|q$. Then we define the twisted $L$-function as the analytic continuation of 
$$L(f,\chi,s):= \sum_{n\geq 1}\frac{\lambda_f(n)\chi(n)}{n^{s+1/2}},$$
(here the Ramanujan conjecture is $|\lambda_f(n)|\leq d(n)n^{1/2}$). By the Birch--Stevens formula (see e.g. \cite[Proposition 6.1]{Nordentoft20.4}), we have for a primitive $\chi$ with $\chi(-1)=\pm 1$;
\begin{equation}\label{eq:BirchStevens} \tau(\overline{\chi})L(f,\chi,s)= \sum_{a (q)}\chi(a) L^{\pm}(f, a/q,s),\end{equation}
where $L^\pm(f, a/q,s)=\frac{1}{2}(L(f, a/q,s)\pm L(f, -a/q,s))$ with 
$$L(f, a/q,s):=\sum_{n\geq 1} \frac{\lambda_f(n)e(na/q)}{n^{s+1/2}},\quad \Re s>1,$$
the additive twist $L$-function satisfying analytic continuation and the functional equation
$$ \gamma_f(s)q^s L(f,a/q,s)= -\gamma_f(1-s)q^{1-s} L(f,-\overline{a}/q,1-s), $$
where $a\overline{a}\equiv 1\modulo q$ and $\gamma_f(s)=\Gamma(s+1/2)(2\pi)^{-s}$ (see e.g. \cite[Proposition 3.3]{Nordentoft20.4}). This implies the functional equation 
 \begin{equation}\label{eq:FEadd}\gamma_f(s)q^{s} L(f,\chi,s)=\eps_\chi  \gamma_f(1-s)q^{1-s} L(f,\overline{\chi},1-s),\end{equation}
where $\eps_\chi= \frac{-\tau(\chi)^2}{q}$ 

\begin{lemma} Let $f\in \mathcal{B}_p$ and $p|q$. Then we have (uniformly in $p$ and $f$) that  
$$ \frac{1}{\varphi^\ast_\pm(p)}\sideset{}{^{\ast,\pm}}\sum_{\chi \, (q)} L(f,\chi,1/2)=1+O_\eps(q^{-1/4+\eps}),$$
where the sum is restricted to primitive characters $\chi$ modulo $q$ with $\chi(-1)=\pm 1$, and $\varphi^*_\pm(q)$ denotes the total number of such characters. 
\end{lemma}
\begin{proof}
By the approximate functional equation \cite[Theorem 5.3]{IwKo} using (\ref{eq:FEadd}), we can write the moment in question as
\begin{align}\label{eq:AFE}
\frac{2}{\varphi^*(q)}\sideset{}{^{\ast,\pm}}\sum_{\chi \, (q)} \sum_{n\geq 1} \frac{\lambda_f(n)}{n}\left(\chi(n) V\left(\frac{n}{q^\lambda}\right)+ \eps_\chi \overline{\chi}(n) V\left(\frac{n}{q^{2-\lambda}}\right)\right),
\end{align} 
for some $0<\lambda<2$ to be chosen and $V: \R_{>0}\rightarrow \R$ a smooth, rapidly decaying function satisfying $V(y)=1+O(y^{1/3})$ as $y\rightarrow 0$ (see \cite[Proposition 5.4]{IwKo}). Now by a simple application of M\"{o}bius inversion we get that for $(n,q)=1$
\begin{align} 
 \label{eq:orth1}\sideset{}{^{\ast,\pm}}\sum_{\chi \, (q)} \chi(n) &=\sum_{d|q} \mu(q/d)\varphi(d)(\delta_{n\equiv 1 \, (d)}\pm \delta_{n\equiv -1 \, (d)})/2  ,\\ 
  \label{eq:orth2} \sideset{}{^{\ast,\pm}}\sum_{\chi \, (q)} \eps_\chi \overline{\chi}(n) &=-\frac{1}{q} \sum_{d|q} \mu(q/d)\varphi(d)\frac{K_2(n;d)\pm K_2(-n;d)}{2},
\end{align}
where $K_2(n;d)= \sum_{xy\equiv n\, (d)} e((x+y)/d)$ denotes the usual $2$-dimensional Kloosterman sum. By Weil's bound and standard estimates, we conclude that for $(n,q)=1$
\begin{equation}
\sideset{}{^{\ast,\pm}}\sum_{\chi \, (q)} \eps_\chi \overline{\chi}(n) \ll_\eps q^{1/2+\eps}. 
\end{equation}
Now for the primary part of the sum (\ref{eq:AFE}) we obtain using the analytic properties of $V$ mentioned above as well as (\ref{eq:orth1})
\begin{align}
&\frac{1}{\varphi^*_\pm(q)}\sideset{}{^{\ast,\pm}}\sum_{\chi \, (q)} \sum_{n\geq 1} \frac{\lambda_f(n)}{n} \chi(n) V\left(\frac{n}{q^\lambda}\right)\\
\label{eq:primaryside}&=1+O_{\eps}\left(\frac{q^{1-\lambda/3}}{\varphi^*_\pm(q)}+\frac{1}{\varphi^*_\pm(q)}\sum_{d|q} \varphi(d) \sum_{\substack{n\equiv \pm 1\, (d)\\ 1<n\ll q^{\lambda+\eps}}} \frac{|\lambda_f(n)|}{n}  \right),
\end{align}
with the main term corresponding to $n=1$. Now using the Ramanujan bound $|\lambda_f(n)|\leq d(n)n^{1/2}\ll_\eps n^{1/2+\eps}$, we get 
\begin{align}
\sum_{\substack{n\equiv \pm 1\, (d)\\ 1<n\ll c^{\lambda+\eps}}} \frac{|\lambda_f(n)|}{n}\ll_\eps \sum_{ 1\leq m \ll q^{\lambda+\eps}/d}\sum_\pm (md\pm 1)^{-1/2+\eps}\ll q^{\lambda/2+\eps}d^{-1},
\end{align}
which bounds the error-term in (\ref{eq:primaryside}) by $O_\eps((q^{1-\lambda/3}+q^{\lambda/2+\eps})/\varphi^\ast(q))$. Similarly by using (\ref{eq:orth2}) we can bound the dual sum by $O_\eps(q^{3/2-\lambda/2+\eps}/\varphi^\ast_\pm(q))$. Choosing $\lambda=3/2$ and recalling that $\varphi^\ast_\pm(q)\gg_\eps q^{1-\eps}$, we get the wanted asymptotic formula.
\end{proof}
\begin{cor} \label{cor:lowerbndadd}There exists an absolute constant $c_0>0$ such that for any prime $p$, $f\in \mathcal{B}_p$ and $\epsilon\in\{\pm\}$, we have
\begin{equation}\label{eq:1stmomentLB}
\frac{1}{p-1}\sum_{0<a  <p} |\langle\{\tfrac{a}{p},\infty\},\omega_f^{\epsilon} \rangle|^2\geq c_0.
\end{equation}
\end{cor}
\begin{proof} First of all recall that the following matrices generate $\Gamma_0(p)$;  
$$T=\begin{psmallmatrix} 1 & 1 \\0 &1  \end{psmallmatrix},\quad \begin{psmallmatrix} a & -(aa^\ast+1)/p \\p &-a^\ast  \end{psmallmatrix}:\quad 0< a<p,$$
where $0<a^\ast<p$ are such that $a\overline{a}\equiv -1\modulo p$ (as is proved in Section \ref{sec:zagier}). Since $\omega_f^{\pm}$ are non-zero cohomology classes vanishing on parabolic elements, we conclude that the left hand side of (\ref{eq:1stmomentLB}) is always non-zero.
  
We now define the Fourier transform of 
$$(\Z/p\Z)^\times \ni a\mapsto \langle\{\tfrac{a}{p},\infty\},\omega_f^{\pm} \rangle, $$
as follows for a Dirichlet character $\chi \modulo p$;
$$\widehat{L}_f^{\pm}(\chi):= \sum_{a\,(p)} \langle\{\tfrac{a}{p},\infty\},\omega_f^{\pm} \rangle \overline{\chi(a)}.$$
For $\chi$ primitive with $\chi(-1)=\pm 1$, we have by the Birch--Stevens formula (\ref{eq:BirchStevens}) that 
$$ \widehat{L}_f^{\pm}(\chi)= -\tau(\overline{\chi})L(f,\chi,1/2). $$ 
Now by Parseval we have 
\begin{align*}
\sum_{0<a <p} |\langle\{\tfrac{a}{p},\infty\},\omega_f^{\pm} \rangle|^2=\frac{1}{p-1}\sum_{\chi\, (p)} |\widehat{L}_f^{\pm}(\chi)|^2 \geq \frac{p}{p-1}\sideset{}{^{\ast,\pm}}\sum_{\chi\, (p)} |L(f,\chi,1/2)|^2.
\end{align*}
Using the previous lemma, we conclude by Cauchy--Schwarz that
$$  \frac{1}{\varphi^\ast_\pm(p)}\sideset{}{^{\ast,\pm}}\sum_{\chi \, (p)} |L(f,\chi,1/2)|^2\geq \left( \frac{1}{\varphi^\ast_\pm(p)} \sideset{}{^{\ast,\pm}}\sum_{\chi \, (p)} L(f,\chi,1/2)\right)^2\geq 1/2, $$
for $p$ large enough. This yields the wanted lower bound since $\varphi_\pm^\ast(p)\asymp p-1$. 
\end{proof}
\begin{remark} Using a subconvexity bound for $L(f\otimes \chi, 1/2)$ as in \cite{MichelVenk10}, the above implies that there is some constant $\delta>0$ such that $L(f\otimes \chi, 1/2)$ is non-vanishing for at least $\gg p^\delta$ of both odd and even characters. This improves on Merel  \cite[Corollaire 2]{Merel09}, who has shown using modular symbol techniques that for $f\in \mathcal{B}_p$ there exists at least one odd and one even character $\chi$ of conductor $p$ such that  $L(f\otimes \chi, 1/2)\neq 0$. One can probably get much better results using mollification.  
\end{remark}
\subsection{Reduction to a counting problem}
We will now use the theta correspondence in a simple form to reduce the sup norm problem to a certain counting problem.  
\begin{lemma}\label{lem:theta}
Let $\omega\in H^1(X_0(p),\R)$ be a cuspidal cohomology class. Then we have 
$$ \sum_{\epsilon\in \{\pm \}}\sum_{f\in \mathcal{B}_p} | \langle v_f^{\epsilon}, \omega\rangle|^2 \ll_\eps p^{-1+\eps}\sum_{n\geq 1} \frac{\frac{1}{p}\sum_{0< x<p}(|\langle T_n \{\tfrac{x}{p}, \infty\}, \omega \rangle|^2+|\langle T_n \{\tfrac{1}{x}, 0\}, \omega \rangle|^2)}{n} e^{-8 n/p},$$
uniformly in $p$, where $\langle \cdot, \cdot \rangle$ denotes the cap product pairing (\ref{eq:cap}) between homology and cohomology and $T_n$ denotes the $n$-th Hecke operator.
\end{lemma}
\begin{proof}
For $0<x<p$ and ${\bm{\epsilon}}=(\epsilon_1,\epsilon_2)$ with $\epsilon_i\in\{\pm 1\}$ we consider the following class in the compactly supported homology 
\begin{align} v_x^{\bm{\epsilon}}:=&(1+\epsilon_1 \iota)(1+\epsilon_2 W_p)\{\tfrac{x}{p} , \infty \}\\
=& \{\tfrac{x}{p} , \infty \}+\epsilon_1 \{-\tfrac{x}{p} , \infty \}+\epsilon_2\{-\tfrac{1}{x},0\}+\epsilon_1\epsilon_2 \{\tfrac{1}{x},0\}\in H_1(X_0(p),\Z),
\end{align}
where $0<\overline{x}<p$ is such that $x\overline{x}\equiv 1\modulo p$. By construction $v_x^{\bm{\epsilon}}$ is contained in the $\epsilon_1$ eigenspace of the involution $\iota$ as in (\ref{eq:iota}) and in the $\epsilon_2$ eigenspace of the Fricke involution $W_p$ as in (\ref{eq:Frickehom}). Associated to $0<x<p$, ${\bm{\epsilon}}\in\{(\pm1,\pm1)\}$ and $\omega\in H^1(X_0(p),\R)$ (suppressed in the notation) we define $g:\Hb\rightarrow \C$ by 
\begin{equation} g(z):=\sum_{n\geq 1} \langle T_n v_x^{\bm{\epsilon}}, \omega \rangle e^{2\pi i nz}.\end{equation}
By expanding in the Hecke basis of cohomology, we obtain
$$  \langle T_n v_x^{\bm{\epsilon}}, \omega \rangle =\sum_{\substack{f\in \mathcal{B}_p:\\ W_p f=\epsilon_2 f}} \lambda_f(n) \langle v_x^{\bm{\epsilon}}, \omega_f^{\epsilon_1}\rangle \langle v_f^{\epsilon_1}, \omega\rangle. $$
By newform theory we know that for $f\in \mathcal{B}_p$ we have $\lambda_f(n)=a_f(n)$ where $a_f(n)$ denotes the Fourier coefficients (at $\infty$) of $f$. Thus we conclude that  
$$ g(z)= \sum_{\substack{f\in \mathcal{B}_p:\\ W_p f=\epsilon_2 f}} \langle v_x^{\bm{\epsilon}}, \omega_f^{\epsilon_1}\rangle \langle v_f^{\epsilon_1}, \omega\rangle f(z),  $$
and in particular $g\in S_2(p)$ is a holomorphic cusp form of weight $2$ for $\Gamma_0(p)$ contained in the $\epsilon_2$ eigenspace of $W_p$. This can be seen as an instance of the theta correspondence as explained above. By orthogonality of Hecke eigenforms, this implies the key identity
\begin{equation} \langle g, g \rangle_\mathrm{Pet} = \sum_{\substack{f\in \mathcal{B}_p:\\ W_p f=\epsilon_2 f}} \langle f, f \rangle_\mathrm{Pet} |\langle v_x^{{\bm{\epsilon}}}, \omega_f^{\epsilon_1}\rangle \langle v_f^{\epsilon_1}, \omega\rangle|^2 ,  \end{equation}   
where $\langle f, g \rangle_\mathrm{Pet}= \int_{Y_0(p)} f(z)\overline{g(z)} dxdy$ denotes the Petterson inner-product on $\mathcal{S}_2(p)$. Recall from  (\ref{eq:funddomainHp}) that the subgroup 
$$\Gamma_0^\ast(p):=\langle W_{p}, \Gamma_0(p)\rangle\subset \PSL_2(\R),$$ 
has $\Gamma_0(p)$ as an index two subgroup and has a fundamental domain contained  in 
$$\{z\in \Hb: |\Re z| \leq 1/2, \Im z \geq \sqrt{3}/(2p)  \}.$$ 
By unfolding and using that $|g|^2$ is invariant under the Fricke involution $W_p$ by construction we arrive at the following bound
\begin{align}\nonumber\sum_{\substack{f\in \mathcal{B}_p:\\ W_p f=\epsilon_2 f}} \langle f, f \rangle_\mathrm{Pet}|\langle v_x^{{\bm{\epsilon}}}, \omega_f^{\epsilon_2}\rangle \langle v_f^{\epsilon_2}, \omega\rangle|^2=2\int_{\Gamma_0^\ast(p)\backslash \Hb} |g(z)|^2 dxdy  &\ll \int_{\sqrt{3}/(2p)}^\infty \int_{-1/2}^{1/2} |g(z)|^2 dxdy\\ 
\nonumber&\ll \sum_{n\geq 1} |\langle T_n v_x^{{\bm{\epsilon}}}, \omega \rangle|^2 \int_{\sqrt{3}/(2p)}^\infty e^{-4\pi n y}  dy\\
\label{eq:mainineq}&\ll \sum_{n\geq 1} \frac{|\langle T_n v_x^{{\bm{\epsilon}}}, \omega \rangle|^2}{n} e^{-2 \sqrt{3}\pi n/p}.
\end{align} 
Clearly we have 
$$|\langle T_n v_x^{{\bm{\epsilon}}}, \omega \rangle|^2 \ll |\langle T_n \{\tfrac{x}{p} , \infty \}, \omega \rangle|^2+|\langle T_n \{-\tfrac{x}{p} , \infty \}, \omega \rangle|^2+|\langle T_n \{\tfrac{1}{x} , \infty \}, \omega \rangle|^2+|\langle T_n \{-\tfrac{1}{x} , \infty \}, \omega \rangle|^2,$$
by linearity of $T_n$. Now we sum over $0<x<p$, $\bm{\epsilon}\in\{(\pm1,\pm1)\}$ and apply Corollary \ref{cor:lowerbndadd} to the left hand side of (\ref{eq:mainineq}). Finally by expressing the Petterson inner product in terms of a special value of an adjoint $L$-function (see e.g. \cite[(8.6)]{PeRi}) and using the lower bounds of Hoffstein--Lockhart \cite{HoffLock94}, we arrive at 
$$ \langle f, f \rangle_\mathrm{Pet}=\frac{pL(\sym^2 f, 1)}{8\pi^3}\gg_\eps p^{1-\eps},$$ 
which yields the wanted bound since $2 \sqrt{3}\pi>8$. 
\end{proof}

\subsection{The counting argument} 
By Lemma \ref{lem:theta} we are reduced to a certain second moment count. We think of this as an analogue of the matrix counts that show up in most approaches to the (arithmetic) sup norm problem (see e.g. \cite{Templier15}). Recall the Ramanujan bound $|\lambda_f(n)|\leq n^{1/2}d(n)$ for Hecke eigenforms $f\in \mathcal{B}_p$. This implies for any class $v\in H_1(X_0(p),\R)$ and $\omega \in H^1(X_0(p),\R)$ we have 
$$\langle T_n v, \omega \rangle\ll_{v,\omega} n^{1/2} d(n),\quad \text{as }n\rightarrow \infty.$$ 
This is however not useful as for us the main point is exactly the dependence on $v$ and $\omega$. Our approach is to use the explicit description of the Hecke operators and then lift the counting from $\Gamma_0(p)^\mathrm{ab}$ to $\Gamma_0(p)$. As a first step, we use geometric coding as in Proposition \ref{prop:geocoding} to obtain the following.
\begin{lemma}\label{lem:geom}
Let $\omega\in H^1(X_0(p),\R)$ be a cuspidal cohomology class 
and $\gamma=\begin{psmallmatrix} a & b \\ c & d\end{psmallmatrix}\in \Gamma_0(p)$ with $c> 0$. Then we have
\begin{equation}\label{eq:geometricbnd} | \langle \{\gamma \infty, \infty \}, \omega \rangle| \ll  \left(\max_{0<x<p} |\langle \{\tfrac{x}{p} ,\infty\}, \omega\rangle|\right)\log  c,\end{equation}
uniformly in $p$.
\end{lemma}
\begin{proof}
Observe that $\{\gamma \infty, \infty \}$ only depends on $a/c \modulo 1$. This means that without changing the homology class we may assume that $a,d$ are integers satisfying $2c\leq |a+d|\leq 4 c $ as well as $-c<a-d<c$. This implies that all entries of $\gamma$ are $O( c )$, and that the half-circle $S_\gamma$ fixed by $\gamma$ intersects the standard fundamental domain $\mathcal{F}_\mathrm{std}$ for $\PSL_2(\Z)$. Recall Zagier's fundamental polygon $\mathcal{F}_\mathrm{Zag}'(p)$ for $\Gamma_0(p)$ defined in Section \ref{sec:zagier} which consists of $\PSL_2(\Z)$-translates of $\mathcal{F}_\mathrm{std}$, as well as the fundamental polygon $\mathcal{F}_\mathrm{Zag}(p):=W_{p} \mathcal{F}_\mathrm{Zag}'(p)$. This latter polygon has the advantage that all entries of the associated side pairing transformations are $\leq p$, whereas for $\mathcal{F}_\mathrm{Zag}'(p)$ the lower left entry can be of magnitude $ p^2$. Notice also that the assumptions on $\gamma$ above insure that the half-circle $S_\gamma$ fixed by $\gamma$ intersects the Siegel domain $\{z\in \Hb: |\Re z|\leq 1/2, \Im z\geq 1\}$ which is contained in $\mathcal{F}_\mathrm{Zag}(p)$.  

We now consider the geometric coding of $\gamma$ with respect to $ \mathcal{F}_\mathrm{Zag}(p)$ as in Corollary \ref{cor:geocoding}. 
This yields an expression 
$$ \gamma= \sigma_1\cdots \sigma_{N},$$
where $\sigma_i\in \mathcal{S}(\mathcal{F}_\mathrm{Zag}(p))$ are side pairing transformations. When considering the class $\{\gamma \infty, \infty \}\in H_1(X_0(p), \Z)$ in the compact homology, we can ignore the parabolic elements among the $\sigma_i$'s. More precisely, if $\sigma'_1,\ldots, \sigma'_\ell$ is the subsequence of $\sigma_1,\ldots ,\sigma_{N}$ consisting of non-parabolic elements, then we have
$$ \{\gamma \infty, \infty \} =\sum_{i=1}^\ell \{ \sigma'_i\infty , \infty \}\in H_1(X_0(p), \Z).$$
Now $\ell$ is exactly the number of intersections between the geodesic from $w$ to $\gamma w$ (for any $w\in \mathcal{F}_\mathrm{Zag}(p)\cap S_\gamma$) and $\Gamma_0(p)$-translates of the sides of $\mathcal{F}_\mathrm{Zag}(p)$ being paired by non-parabolic elements. This is equal to the number of intersections between the geodesic from $w$ to $\gamma' w$ (where $\gamma'=W_{p}\gamma W_{p}^{-1}$ and $w\in \mathcal{F}_\mathrm{Zag}'(p)\cap S_{\gamma'}$) and $\Gamma_0(p)$-translates of the non-parabolic sides of $\mathcal{F}_\mathrm{Zag}'(p)$ (this follows by conjugating everything which preserves parabolicity). Recall that for $\mathcal{F}_\mathrm{Zag}'(p)$, all sides containing a cusp are paired by parabolic elements. Since $\mathcal{F}_\mathrm{Zag}'(p)$ consists of $\PSL_2(\Z)$-translates of $\mathcal{F}_\mathrm{std}$, we can bound $\ell$ by the number of intersections between the geodesic from $w$ to $\gamma' w$ (for $w\in \mathcal{F}_\mathrm{std}\cap S_{\gamma'}$) and $\PSL_2(\Z)$-translates of the non-parabolic side of $\mathcal{F}_\mathrm{std}$ i.e. the arc 
$$\{z\in \Hb: |z|=1, -1/2<\Re z< 1/2\}.$$ 
It now follows from a result of Eichler \cite[Satz 1]{Eichler65} that 
$$ \ell \ll \log ((a')^2+(b')^2+(c')^2+(d')^2) \ll \log  c , $$
where $a',b',c',d'$ are the entries of $\gamma'$, which by the assumptions on the entries of $\gamma$ are all $O( p c)= O(c ^2)$. 
By definition of $\mathcal{F}_\mathrm{Zag}(p)$ we have for all $1\leq i\leq \ell $ that $\sigma'_i\infty$ is of the form $\tfrac{x}{p}$ with $0<x<p$. Thus we conclude  
$$ |\langle \{\gamma \infty, \infty \},\omega \rangle| \ll  \ell \left(\max_{0<x<p} |\langle \{\tfrac{x}{p} ,\infty\}, \omega\rangle|\right)\ll 
\log c \left(\max_{0<x<p}|\langle \{\tfrac{x}{p} ,\infty\}, \omega\rangle|\right), $$
as wanted.
\end{proof}
This implies the following result for basic bases. Recall the definition of the cuspidal projection operator $\mathbb{P}_\mathrm{cusp}$ in (\ref{eq:franke1}).
\begin{cor}\label{cor:bndbasic}
Let ${B}$ be a basic basis of level $p$. Then for $\gamma=\begin{psmallmatrix} a & b \\ c & d\end{psmallmatrix}\in \Gamma_0(p)$ with $c>0$ and $\omega \in  B^{\ast}$ we have
$$  |\langle \{\gamma \infty,\infty\}, \mathbb{P}_\mathrm{cusp} \, \omega\rangle|\ll \log c, $$
uniformly in $p$.
\end{cor}
\begin{proof}
By Proposition \ref{lem:geom} we are reduced to proving that 
$$ \max_{0<x<p}|\langle \{\tfrac{x}{p} ,\infty\}, \mathbb{P}_\mathrm{cusp}\omega\rangle|\ll 1,$$
for $\omega \in  B^{\ast}$. Recall that we have
$$ \langle \{\tfrac{x}{p} ,\infty\}, \mathbb{P}_\mathrm{cusp}\omega\rangle= \langle \{\gamma_x z  ,z\}, \omega\rangle-\langle v_{E}(p), \omega\rangle\langle \{\gamma_x z  ,z\}, \omega_{E}(p)\rangle, $$
where $\gamma_x:= \begin{psmallmatrix} x & -(xx^\ast+1)/p \\ p& -x^\ast \end{psmallmatrix}$ with $0<x^\ast<p$ such that $xx^\ast\equiv -1\modulo p$ and $z\in \Hb $ is arbitrary. Now by Corollary \ref{lem:relations}, we have
$$\langle v_{E}(p), \omega\rangle\ll 1,\quad   \langle \{\gamma_x z  ,z\}, \omega\rangle \ll 1,$$
and by (\ref{eq:BoundEisenstein}) we have 
 $$\langle \{\gamma_x z  ,z\}, \omega_{E}(p)\rangle\ll \frac{p|x-x^\ast|+p^2}{p^2}\ll 1,$$
 which yields the wanted bound.
\end{proof}

\subsection{Proof of Theorem \ref{thm:maincontr}}

Combining all of the above  we are now ready to prove  our sup norm bounds.

\begin{proof}[Proof of Theorem \ref{thm:maincontr}] 
Let $B$ be a basic basis of level $p$. By the definition of the Hecke operators acting on homology (\ref{eq:Tn}) we have for $\omega\in  B^{\ast}$ and $n\ll p^{1+\eps}$ by Corollary \ref{cor:bndbasic}\begin{align}
\nonumber |\langle T_n \{\tfrac{x}{p} , \infty \}, \mathbb{P}_\mathrm{cusp}\, \omega \rangle|
\nonumber &\leq \sum_{\substack{ad=n,\\ (a,p)=1}}\sum_{0\leq b <d}  \left|\left\langle \left\{ \frac{ax+bp}{pd}, \infty  \right\}, \mathbb{P}_\mathrm{cusp}\,\omega \right\rangle\right|\\
\label{eq:redsupnorm} & \ll  \sum_{\substack{ad=n,\\ (a,p)=1}}\sum_{0\leq b <d}  \log pd \ll_\eps p^{1+\eps}  .
\end{align}
There we are using that $(ax+bp,p)=1$ meaning that $\frac{ ax+bp}{pd}$ is of the shape $\gamma \infty$ for $\gamma \in \Gamma_0(p)$ with left lower entry of size $O(pd)$. By changing contours we have the following equality of compactly supported homology classes;
$$\{\tfrac{1}{x},0\}=\{\begin{psmallmatrix} \overline{x}& 1\\ 1-x\overline{x}&x\end{psmallmatrix}0,0\} =\{\begin{psmallmatrix} \overline{x}& 1\\ 1-x\overline{x}&x\end{psmallmatrix}\infty,\infty\}=\{ \tfrac{-\overline{x}}{x\overline{x}-1} ,\infty\},\quad 0<\overline{x}<p: x\overline{x}\equiv 1\modulo p.$$
A similar argument as above again using Corollary \ref{cor:bndbasic} yields  
$$|\langle T_n \{\tfrac{1}{x} , \infty \}, \mathbb{P}_\mathrm{cusp}\, \omega \rangle|\ll_\eps p^{1+\eps}, \quad n\leq p^{1+\eps}, \omega\in B^\ast.$$
Thus by Lemma \ref{lem:theta} we conclude as wanted 
\begin{align*}
\sum_{f\in \mathcal{B}_p} | \langle v_f^{\pm}, \omega\rangle|^2 \ll_\eps p^{-1+\eps}\sum_{n\geq 1} \frac{p^{2+\eps}}{n} e^{-8 n/p} 
\ll_\eps p^{1+\eps} ,\quad \omega\in  B^{\ast}.\end{align*}
\end{proof} 
 \section{Proof of level aspect version}\label{sec:levelp}
   
Using the second moment bound from Theorem \ref{thm:maincontr} we are ready to prove our main result. 
\begin{thm}\label{thm:maingeneral} Let $p$ be prime and let $K$ be a real quadratic field of discriminant $d_K $ with no unit of norm $-1$ such that $p$ splits in $K$ with $p\mathcal{O}_K=\mathfrak{p}_1\mathfrak{p}_2$ and $\mathfrak{p}_1\notin H$ where $H=(\Cl^+_K)^2$. 

For ${B}$ a basic basis of level $p$ and $\omega\in  B^{\ast}$, we have   
\begin{equation}\label{eq:thm1general2}  \frac{\sum_{A\in H} \langle [\mathcal{C}_{A}(p)], \omega \rangle}{|\sum_{A\in H} \langle [\mathcal{C}_{A}(p)], \omega_{E}(p) \rangle|}= -\langle v_{E}(p), \omega\rangle+O_{\eps}(d_K^{-1/12+\eps}p^{2+\eps}), \end{equation}
where $v_{E}(p),\omega_{E}(p)$ denotes the Eisenstein classes in homology and cohomology as defined in (\ref{eq:eisensteinclasshom}) and (\ref{eq:eisensteinclass}), respectively.
\end{thm} 

\begin{proof}
Starting from (\ref{eq:Wald1}), we want to bound the cuspidal contribution on the righthand side, which by (\ref{eq:Wald2}) can be expressed in terms of $L$-functions. Recall that class group characters $\chi$ such that $\chi_{|(\Cl_K^+)^2}$ (i.e. a genus characters) correspond to factorizations $d_1d_2=d_K$ in terms of fundamental discriminants, and we have
$$L(f\otimes \pi_\chi,s)=L(f\otimes \chi_{d_1},s)L(f\otimes \chi_{d_2},s),$$
where $\chi_{d_i}=(\tfrac{d_i}{\cdot})$ are quadratic characters. Now we apply Cauchy--Schwarz followed by H\"{o}lder's inequality with exponents $(3,3,3)$ combined with Theorem \ref{thm:maincontr} as well as the third moment bound of Petrow and Young \cite[Theorem 1]{PeYo19};
\begin{align}
\nonumber &\left|\sum_{A\in H} \langle [\mathcal{C}_{A}(p)], \mathbb{P}_\mathrm{cusp}\, \omega \rangle\right|^2\\
 \nonumber&\ll d_K^{1/2+\eps}\left(\sum_{f\in \mathcal{B}_p, \pm} |\langle v_f^{\pm}, \omega\rangle |^2\right) \left(\sum_{d_1d_2=d_K } \sum_{f\in \mathcal{B}_p}L(f\otimes \chi_{d_1},1/2) L(f\otimes \chi_{d_2},1/2)\right)\\
 \nonumber &\ll d_K^{1/2+\eps} \left(\sum_{f\in \mathcal{B}_p, \pm} |\langle v_f^{\pm}, \omega\rangle |^{2}\right)  \\
\label{eq:boundcuspidal}  &\qquad \qquad \qquad \times\sum_{d_1d_2=d_K }  p^{1/3}\left(\sum_{f\in \mathcal{B}_p} L(f\otimes \chi_{d_1},1/2)^3 \right)^{1/3}  \left(\sum_{f\in \mathcal{B}_p} L(f\otimes \chi_{d_2},1/2)^3 \right)^{1/3}\\
\nonumber &\ll_\eps p^{2+\eps} d_K^{5/6+\eps} , 
\end{align}
for all $\omega\in B^{\ast}$ (using here also that $L(f\otimes \chi_{d_i},1/2)\geq 0$). It follows from Proposition \ref{prop:lowerbound} that the Eisenstein contribution on the righthand side of (\ref{eq:Wald1}) is $\gg_\eps d_K^{1/2-\eps} p^{-1}$. Inserting all of this into (\ref{eq:Wald1}) yields as wanted. 
\end{proof}

Recall that given a basis $B\subset V_p=H_1(Y_0(p),\R)$ we get an associated isomorphism $V_p\cong \R^{2g+1}$ by mapping $B$ to the standard basis of $\R^{2g+1}$. By pulling back the $L^r$-norm with respect to the standard basis of $\R^{2g+1}$ we get the following norm on $V_p$
\begin{equation}\label{eq:normsB}|\!|v |\!|_{B,r}:=\left(\sum_{\omega\in  B^{\ast}}|\langle v,\omega  \rangle|^r\right)^{1/r}\text{ for } 1\leq r<\infty, \quad |\!|v |\!|_{B,\infty}:=\sup_{\omega\in  B^{\ast}}|\langle v,\omega  \rangle|, \end{equation}
with $B^\ast\subset H^1(Y_0(p),\R)$ the dual basis of $B$. Combining Theorem \ref{thm:maingeneral} and equations (\ref{eq:supnormcalc}), (\ref{eq:translationtonorm}) we arrive at the following (we will skip the details).

\begin{thm}\label{thm:maingen} 
Let $p$ be prime and let $K$ be a real quadratic field of discriminant $d_K $ with no unit of norm $-1$ such that $p$ splits in $K$ with $p\mathcal{O}_K=\mathfrak{p}_1\mathfrak{p}_2$ and $\mathfrak{p}_1\notin H$ where $H=(\Cl^+_K)^2$. 

For $B$ a basic basis of level $p$ we have   
\begin{equation}\label{eq:thm1bsc}  \left|\! \left|\frac{\sum_{A\in H} \langle [\mathcal{C}_{A}(p)]\rangle}{|\!|\sum_{A\in H} \langle [\mathcal{C}_{A}(p)]\rangle|\!|_{B,\infty}}+v_{E}(p)\right|\! \right|_{B,\infty} \ll_{\eps} d_K^{-1/12+\eps}p^{2+\eps} . \end{equation}
\end{thm} 

\section{Applications}\label{sec:applproofs}
In this section we will present certain applications of Theorem \ref{thm:maingen}.
\subsection{A group theoretic application}\label{sec:group}
Recall the independent generators of $\Gamma_0(p)$ described in Section \ref{sec:kul} coming from special fundamental polygons. Given a real quadratic field $K$ of discriminant $d_K $ such that $p$ splits in $K$, the oriented closed geodesics $\mathcal{C}_{A}(p)$ associated to $A\in \Cl_K^+$ corresponds to a conjugacy class of matrices inside $\Gamma_0(p)$ (given by the matrices $\gamma_Q$ defined in (\ref{eq:gammaQ}) where $Q\in \mathcal{Q}_{K,p}$ runs through integral binary quadratic form of discriminant $d_K $ and level $p$ corresponding to $A$ under the identification (\ref{eq:Gaussiso})). We will apply our results to understand the representation of the matrices $\gamma_Q$ in terms of the independent generators (\ref{eq:basisT}) (\ref{eq:basise2}), (\ref{eq:basise3}), (\ref{eq:basisK2}) coming from a special fundamental domain of $\Gamma_0(p)$. 
\begin{cor}\label{cor:groupgen}
Let $p$ be prime and let $\mathcal{P}(p)$ be a special fundamental polygon for $\Gamma_0(p)$. Consider a real quadratic field $K$ of discriminant $d_K $ with no unit of norm $-1$ such that $p$ splits in $K$ with $p\mathcal{O}_K=\mathfrak{p}_1\mathfrak{p}_2$ and such that $\mathfrak{p}_1 \notin (\Cl_K^+)^2$. Then for $d_K \gg_\eps p^{24+\eps}$ there is some class $A\in (\Cl_K^+)^2$ such that \underline{non} of the matrices $\gamma_Q\in \Gamma_0(p)$ with $Q\in \mathcal{Q}_{K,p}$ corresponding to $A$ are contained in the subgroup generated by the matrices in (\ref{eq:basise2}), (\ref{eq:basise3}), (\ref{eq:basisK2}).
\end{cor}
\begin{proof}
Let $B_\mathrm{sp}\subset H_1(Y_0(p),\R)$ be a basis obtained from the side pairing transformations of $\mathcal{P}(p)$. Note that $v_{E}(p)\in B_\mathrm{sp}$ and let $\omega_0\in  B_\mathrm{sp}^{\ast}$ be characterized by $\langle v_{E}(p),\omega_0\rangle=1$. 
Let ${B}$ be a basic basis of level $p$ containing $v_{E}(p)$. Then by expanding in this basis we get by Theorem \ref{thm:maingen} that

\begin{align}
\nonumber\left\langle \frac{\sum_{A\in H}[\mathcal{C}_{A}(p)]}{|\!|\sum_{A\in H}[\mathcal{C}_{A}(p)]|\!|} , \omega_0 \right\rangle&=-\left\langle \frac{v_{E}(p)}{|\!|v_{E}(p)|\!|}, \omega_0 \right\rangle+ \left\langle \frac{\sum_{A\in H}[\mathcal{C}_{A}(p)]}{|\!|\sum_{A\in H}[\mathcal{C}_{A}(p)]|\!|}+\frac{v_{E}(p)}{|\!|v_{E}(p)|\!|}, \omega_0 \right\rangle  \\
\nonumber &=-1+\sum_{v\in B}\left\langle v, \omega_0 \right\rangle \left\langle \frac{\sum_{A\in H}[\mathcal{C}_{A}(p)]}{|\!|\sum_{A\in H}[\mathcal{C}_{A}(p)]|\!|}+\frac{v_{E}(p)}{|\!|v_{E}(p)|\!|}, v^\ast  \right\rangle \\
\label{eq:grouptheory}&= -1+ O_\eps\left( d_K^{-1/12+\eps}p^{2+\eps}\sum_{v\in B}|\langle v, \omega_0 \rangle |\right),
 \end{align}
where $|\!|\cdot |\!|=|\!|\cdot |\!|_{B,\infty}$ and $v^\ast\in B^\ast$ is characterized by $\langle v, v^\ast\rangle=1$ 
For $v\in {B}-\{v_{E}(p)\}$, let $\gamma_v\in \mathcal{S}(\mathcal{F}_\mathrm{Zag}(p))$ be such that 
$$v=\{z,\gamma_v z\}\in H_1(Y_0(p),\R),$$ using the notation (\ref{eq:def}) (note that the lower left entry of $\gamma_v$ is equal to $p$). Now for $Y>0$ consider a curve $c_Y: [0,1] \rightarrow \Hb$ connecting the three points 
$$\gamma_v^{-1} \infty +iY ,\quad  \gamma_v \infty +iY,\quad \gamma_v(\gamma_v^{-1} \infty +iY)=\gamma_v\infty+ip^{-2}Y^{-1},$$
by (Euclidean) straight line. We assume that $Y$ is large enough so that $$ \{x +iY:\min (\gamma_v\infty,\gamma_v^{-1} \infty) < x < \max(\gamma_v\infty,\gamma_v^{-1} \infty)\}\subset \mathcal{P}(p).$$
Observe that by Corollary \ref{cor:geocoding} the quantity $|\langle v, \omega_0 \rangle |$ is bounded by the number of intersection between $c_Y$ and $\Gamma_0(p)$-translates of $\{iy:y>0\}$ (which is a side of the special fundamental polygon $\mathcal{P}(p)$). Since $0<\gamma_v\infty, \gamma_v^{-1}\infty<1$ there is no such intersection for the horizontal segment of $c_Y$. Intersections with the vertical segment corresponds to integers $a,b,c,d$ such that
\begin{equation}\label{eq:verticalsegment}p|c,\qquad ad-bc=1,\qquad\text{and}\quad  \frac{a}{c}<\gamma_v\infty <\frac{b}{d}\text{ or }\frac{b}{d}<\gamma_v\infty <\frac{a}{c}.\end{equation}
In particular we have $\gamma_v\neq \frac{a}{c}$ which implies 
$$\left|\gamma_v\infty-\frac{a}{c}\right|\geq \frac{1}{|c|},\qquad \left|\frac{b}{d}-\frac{a}{c}\right|=\frac{1}{|cd|},$$ 
but this contradicts (\ref{eq:verticalsegment}). This implies that $\langle v, \omega_0 \rangle=0$ for $v\in {B}-\{v_{E}(p)\}$ which by (\ref{eq:grouptheory}) yields the wanted error-term.   
\end{proof}
From the above Corollary \ref{cor:group} follows in the case where $K$ has wide class number one and narrow class number two by writing out explicitly the matrices $\gamma_Q$ associated to $Q\in \mathcal{Q}_{K,p}$ as in (\ref{eq:gammaQ}) and recalling that since $K$ does not have a unit of norm $-1$ the condition $\mathfrak{p}_1 \notin (\Cl_K^+)^2$ is equivalent to $\mathfrak{p}_1$ not having a generator of positive norm. 

\subsection{An application to modular forms}
Recall the following definition mentioned in the introduction for a modular form $f\in \mathcal{M}_2(p)$:
$$M_{f}:=\inf \{c\geq 0: |a_f(n)|\leq c \sigma_1(n) ,\, n\geq 1\}<\infty, $$  
where $a_f(n)$ denotes the Fourier coefficients of $f$ (at $\infty$) and $\sigma_1(n)=\sum_{d|n}d$.  
We have the following non-vanishing result for cycle integrals of modular forms.
 
\begin{cor}\label{cor:modular2}
Let $p$ be prime and let $f\in \mathcal{M}_2(p)$ be a holomorphic modular form of weight $2$ and level $p$ with constant Fourier coefficient equal to $a_f(0)=1$.
Consider a real quadratic field $K$ of discriminant $d_K $  with no unit of norm $-1$ such that $p$ splits in $K$ with $p\mathcal{O}_K=\mathfrak{p}_1\mathfrak{p}_2$ and  $\mathfrak{p}_1 \notin (\Cl_K^+)^2$. Then for $d_K \gg_\eps (M_{f})^{12+\eps} p^{48+\eps}$ there is some $A\in (\Cl_K^+)^2$ such that   
$$\langle [\mathcal{C}_{A}(p)], f(z)dz\rangle= \int_{\mathcal{C}_{A}(p)} f(z)dz\neq 0.$$ 
\end{cor} 
\begin{proof}
Let
$$f(z)dz= \omega_+ +i\omega_-\in H^1(Y_0(p),\R)\oplus i H^1(X_0(p),\R),$$
be the cohomology class (with complex coefficients) associated to $f$. By expanding in a basic basis $B\subset H_1(Y_0(p),\R)$ as in (\ref{eq:grouptheory}) and using Cauchy--Schwarz followed up by  (\ref{eq:thm1bsc}) we conclude that 
\begin{align*}
\left|\left\langle \frac{\sum_{A\in H}[\mathcal{C}_{A}(p)]}{|\!|\sum_{A\in H}[\mathcal{C}_{A}(p)]|\!|_{B,\infty}}-v_{E}(p) , \omega_\pm \right\rangle\right|\ll_\eps p^{3+\eps} d_K^{-1/12+\eps}\sup_{v\in {B}}|\langle v, \omega_\pm \rangle |.
 \end{align*}
Now by the assumptions on the Fourier expansion of $f$, we get:
\begin{align*} f(x+iy)&\ll 1+ M_{f}\sum_{n\geq 1}\sigma_1(n) e^{-2\pi yn}\ll_\eps 1+M_{f}\int_1^\infty t^{1+\eps} e^{-2\pi ty} \frac{dt}{t}\ll_\eps \mathrm{max} (1, M_{f} y^{-1-\eps}). \end{align*}
 In order to bound the quantities $\langle v, \omega_\pm\rangle$ for $v\in {B}$, we recall that we can write $v=\{\gamma z, z\}$ for $z\in \Hb$ (which we may choose freely) and some matrix $\gamma=\begin{psmallmatrix} a & b \\ c & d \end{psmallmatrix}\in \mathcal{S}(\mathcal{F}_\mathrm{Zag}(p))$ as in Section \ref{sec:zagier}. In particular, we observe that by construction we have $a,b,c,d\ll p$ and $ c\in\{ 0,p\}$. If $c=0$, then clearly $ \langle v, \omega_\pm \rangle\in\{0,1\}$. If $c=p$, we make the following convenient choice $z=\frac{-d}{p}+\frac{i}{p}$ satisfying $\gamma z=\frac{a}{p}+\frac{i}{p}$. This gives 
 $$ \langle v, \omega_\pm \rangle= \int_{\tfrac{a}{p}}^{-d/p} \frac{f(x+ip^{-1})\pm \overline{f(x+ip^{-1})}}{2i^{(1\mp 1)/2}} dx\ll_\eps M_{f} \frac{|a+d|}{p}p^{1+\eps}\ll_\eps M_{f} p^{1+\eps},$$
which we plug into the above. Since $\langle v_{E}(p),\omega_+ \rangle=1$ and $\langle v_{E}(p),\omega_- \rangle=0$ we conclude that $\langle\sum_{A\in H}[\mathcal{C}_{A}(p)] , f(z)dz \rangle$ is indeed non-vanishing for $d_K \gg_\eps (M_{f})^{12+\eps} p^{48+\eps}$, as wanted. 
\end{proof}
 Now Corollary \ref{cor:modular} follows from the above in the case of wide class number one and narrow class number two.
\bibliography{/Users/asbjornnordentoft/Documents/MatematikSamlet/Bib-tex/mybib}
\bibliographystyle{alpha}
\end{document}